\documentclass[reqno]{amsart}
\usepackage{amsmath,amssymb,cite}
\usepackage[mathscr]{euscript}
\theoremstyle{plain}
\newtheorem{theorem}{Theorem}[section]
\newtheorem{lem}[theorem]{Lemma}{\bf}{\it}
\newtheorem{prop}[theorem]{Proposition}{\bf}{\it}
{\bf}{\it}
\newtheorem{defin}[theorem]{Definition}{\bf}{\rm}
\newtheorem{assumption}[theorem]{Assumption}{\bf}{\rm}
\newtheorem{rem}[theorem]{Remark}{\bf}{\rm}
\numberwithin{equation}{section}
\numberwithin{theorem}{section}

\newcommand{\res}{\mathop{\hbox{\vrule height 7pt width .5pt depth 0pt\vrule height .5pt width 6pt depth 0pt}}\nolimits}
\newcommand{\mc}[1]{{\mathcal #1}}

\newcommand{\bs}[1]{{\boldsymbol #1}}
\newcommand{\bb}[1]{{\mathbb #1}}
\newcommand{\rme}{\mathrm{e}}
\newcommand{\rmi}{\mathrm{i}}
\newcommand{\rmd}{\mathrm{d}}
\newcommand{\id}{{1 \mskip -5mu {\rm I}}}
\newcommand{\eps}{\varepsilon}
\newcommand{\supp}{\mathop{\rm supp}\nolimits}
\newcommand{\essup}{\mathop{\rm ess\,sup}}

\title[LDP for stochastic PDE approximation of the mean curvature flow]
{Stochastic Allen-Cahn approximation of the mean curvature flow: large deviations upper bound}

\author[L.\ Bertini]{Lorenzo Bertini}
\address{Lorenzo Bertini \hfill\break \indent
   Dipartimento di Matematica, Sapienza Universit\`a di Roma
   \hfill\break \indent
   P.le Aldo Moro 5, 00185 Rome, Italy}
 \email{bertini@mat.uniroma1.it}

\author[P.\ Butt\`a]{Paolo Butt\`a}
\address{Paolo Butt\`a\hfill\break \indent
   Dipartimento di Matematica, Sapienza Universit\`a di Roma 
   \hfill\break \indent
   P.le Aldo Moro 5, 00185 Rome, Italy}
 \email{butta@mat.uniroma1.it}

\author[A.\ Pisante]{Adriano Pisante}
\address{Adriano Pisante \hfill\break \indent
   Dipartimento di Matematica, Sapienza Universit\`a di Roma 
   \hfill\break \indent
   P.le Aldo Moro 5, 00185 Rome, Italy}
 \email{pisante@mat.uniroma1.it}

\begin{document}

\begin{abstract}
Consider the Allen-Cahn equation on the $d$-dimensional torus, $d=2,3$, in the sharp interface limit. As it is well known, the limiting dynamics is described by the motion by mean curvature of the interface between the two stable phases. Here, we analyze a stochastic perturbation of the Allen-Cahn equation and describe its large deviations asymptotics in a joint sharp interface and small noise limit. Relying on previous results on the variational convergence of the action functional, we prove the large deviations upper bound. The corresponding rate function is finite only when there exists a time evolving interface of codimension one between the two stable phases. The zero level set of this rate function is given by the evolution by mean curvature in the sense of Brakke. Finally, the rate function can be written in terms of the sum of two non-negative quantities: the first measures how much the velocity of the interface deviates from its mean curvature, while the second is due to the possible occurrence of nucleation events. 
\end{abstract}

\maketitle
\thispagestyle{empty}
%\tableofcontents

\section{Introduction}
\label{sec:1}

The van der Waals theory of phase transitions \cite{CH,vdW}  is based on the excess free energy functional,
\begin{equation}
\label{F0}
\mc F (u) := \int\!\Big[ \frac12 |\nabla u|^2 + W(u) \Big]\,\rmd x\;, 
\end{equation}
where $u\colon \bb R^d \to \bb R$ is the local order parameter and $W\colon \bb R\to [0,+\infty)$ is a smooth, symmetric, double well potential whose minimum value, chosen to be zero, is attained at, say, $u_\pm$. The constant functions $u(x) =u_\pm$ are interpreted as the pure phases of the system. The potential $W(u)$ represents the excess ``mean field'' free energy density of the homogenous state $u$ with respect to the pure phases $u_\pm$, while the gradient term in \eqref{F0} penalizes spatial variations of $u$. 

The sharp interface limit of \eqref{F0} has been analyzed in \cite{Modica} and extensively studied afterwards, see \cite{Alberti} for a review. The limit of the (properly rescaled) free energy turns out to be finite only if $u$ is a function of bounded variation taking values in $\{u_-,u_+\}$. For $u$ in this set, the limiting functional is given by $\tau\, {\mc H}^{d-1}(\mc S_u)$, where $\mc S_u$ denotes the jump set of $u$ and ${\mc H}^{d-1}(\mc S_u)$ is its $(d-1)$-dimensional Hausdorff measure. The \emph{surface tension} $\tau$ is given by
\begin{equation}
\label{tau} 
\tau = \int_{u_-}^{u_+}\! \sqrt{2W(s)} \, \rmd s\;.
\end{equation}
We note that $\tau$ can also be characterized as the minimum value of the one-dimensio\-nal excess free energy $\mc F$ in \eqref{F0} with the constraint $u(x) \to u_\pm$ as $x\to \pm \infty$.

After the pioneering paper \cite{AC}, the $L^2$-gradient flow of \eqref{F0}, i.e., the semi-linear parabolic equation,
\begin{equation}
  \label{1.1}
  \partial_t u = \Delta u - W'(u)\;,
\end{equation}
has become a basic model in the kinetics of phase separation and interface dynamics for systems with a non-conserved order parameter $u=u_t(x)$. 

Consider the evolution induced by \eqref{1.1} under diffusive rescaling of time and space. For suitably prepared initial data, which approach a sharp interface between the pure phases $u_\pm$, the asymptotics of the solution to \eqref{1.1} is described by the motion by mean curvature of the interface. This has been proven in \cite{Ilmanen} in the weak formulation of the mean curvature flow in terms of Brakke motions \cite{Brakke}, see also, e.g., \cite{BSS,ESS} for similar results in the framework of the level-set formulation. 

From both a phenomenological and a conceptual viewpoint, the addition of a random forcing term to \eqref{1.1}, that models the thermal fluctuation in the system, appears quite natural. Assuming this forcing to be Gaussian and translation covariant, we are led to consider the stochastic partial differential equation,
\begin{equation}
  \label{1.2}
  \partial_t u =\Delta u - W'(u) + \sqrt{2 \lambda} \, \eta^\gamma\;,
\end{equation}
where $\lambda>0$ measures the strength of the noise and $\eta^\gamma$ is a mean zero Gaussian space-time noise, that is white in time and whose space correlation is of order $\gamma$, e.g., 
\begin{equation}
  \label{1.3}
  \bb E \big( \eta^\gamma(t,x) \eta^\gamma(t',x') \big) = \delta(t-t') 
  \, \imath_\gamma(x-x')\;, \qquad \imath_\gamma(x) =
  \gamma^{-d}\imath(\gamma^{-1}x)\;,  
\end{equation}
where $\imath$ is a smooth positive function on $\bb R^d$ with compact support. For $\gamma>0$ the well-posedness and regularity properties of \eqref{1.2} in space dimension $d\le 3$ are discussed in \cite{BBP1}.

We understand that for $\gamma=0$ the process $\eta^\gamma$ is the space-time white noise. In this case - in space dimension $d>1$ - the well-posedness of \eqref{1.2} becomes a major issue and a proper renormalization of the non linear term $W'$ is needed.  In dimension $d=2$, when $W$ is a polynomial, this renormalization amounts to the Wick ordering \cite{AR,DD,JM,MW}. In dimension $d=3$, the renormalization of the non linearity is more involved; for a quartic potential $W$, existence and uniqueness of local-in-time solutions is proven in \cite{Hairer} and, more recently, global well-posedness has been obtained in \cite{MW1}.
 
Consider \eqref{1.2} in a bounded volume $\Lambda$ on the time interval $[0,T]$.  The corresponding large deviations are analyzed in \cite{CF} in the joint limit $\lambda\to 0$, $\gamma\to 0$. Under suitable conditions on these sequences, it is shown that the rate function is given, as it can be guessed from the Freidlin-Wentzell theory for finite dimension diffusion processes \cite{FW}, by
\begin{equation}
  \label{ifw}
  J (u) =\frac 14 \int_0^T\!\int_{\Lambda}  
  \big[ \partial_t u - \big( \Delta u - W'(u)\big) \big]^2\, \rmd x\,  \rmd t\;.
\end{equation}
Informally, denoting by $u^{\lambda,\gamma}$ the solution to \eqref{1.2}, the large deviations statement corresponds to the asymptotics,
\[
\bb P(u^{\lambda,\gamma}\in B) \asymp \exp\{-\lambda^{-1}\inf_{u\in B}J(u)\}\;.
\]
In space dimension $d\le 3$, the same rate function is obtained in the case of space-time white noise, that is when the parameter $\gamma$ is set equal to zero from the beginning. This has been proven in \cite{FJ} for $d=1$, in \cite{JM2} for $d=2$ (to be precise, it is there considered a non-local version of \eqref{1.2}), and \cite{HW} for $d=2,3$.  As we mentioned above, in space dimension $d=2,3$, the reaction term $W'$ has to be renormalized by subtracting infinite terms. On the other hand, the rate function is \eqref{ifw} \emph{without} any renormalization on $W'$. 
Very loosely, the underlying reason is the following. The large deviations principle is established in a weak topology and, although the added counter-terms are infinite (diverging as $\gamma\to 0$ if the noise is mollified as in \eqref{1.3}), they are multiplied by $\lambda$ and therefore irrelevant for the large deviations. 

The purpose of the present paper is to analyze the large deviations asymptotics of \eqref{1.2} under diffusive rescaling of space and time, i.e., in the sharp interface (singular) limit. By denoting with $\eps$ the scaling parameter and redefining the parameters $\lambda$ and $\gamma$, we thus consider the stochastic equation,
\begin{equation}
  \label{1.3b}
  \partial_t u = \Delta u - \frac{1}{\eps^2} W'(u) 
  + \sqrt{2 \lambda} \, \eta^\gamma\;,
\end{equation}
on a bounded volume that, to avoid the somewhat delicate issue of boundary conditions, we choose to be the $d$--dimensional torus. We are now interested in the joint limit $\eps,\lambda,\gamma\to 0$. 

To pursue the above program, one possibility is to take first the limit $\lambda,\gamma\to 0$ and then $\eps\to 0$. In view of the result in \cite{CF}, one is then led to analyze the variational convergence, more precisely the $\Gamma$-convergence \cite{DalMaso}, of the sequence of action functionals $(I_\eps)$ defined by
\begin{equation}
  \label{ifwe}
  I_\eps(u) =\frac 14 \eps \int_0^T\!\int_{\Lambda}
  \Big[ \partial_t u - \Big( \Delta u - \frac 1{\eps^2} W'(u)\Big)
  \Big]^2\, \rmd x\,  \rmd t\;,
\end{equation}
in which the pre-factor $\eps$ has been inserted to have a finite limit. The problem of the variational convergence of $(I_\eps)$ has been analyzed in \cite{KRT,KORV}, precisely with the motivation of the large deviations asymptotics of the stochastic Allen-Cahn equation, and in greater detail in \cite{MR}. The precise definition of the limiting functional requires tools from geometric measure theory and it is deferred to the next section. Here, we just give a heuristic description of the results obtained in \cite{KORV,MR}. Assume $d\le 3$. The limiting functional is finite only if $u$ takes value in $\{u_\pm\}$ and in the simplest case of interfaces with multiplicity one is given by
\begin{equation}
  \label{i0}
  I_0(u) = \frac {\tau}4 \int_0^T\! \int_{\Sigma_t} \big| \nu_t
  -H_t\big|^2  \, \rmd \mc H^{d-1} \, \rmd t +I_\mathrm{nucl} (u)\;, 
\end{equation}
where $\tau$ is defined in \eqref{tau}, $\Sigma_t$ is the boundary of $\{x\colon u_t(x) = u_+\}$, $\nu_t$ is the normal velocity of this set, and $H_t$ its mean curvature vector. Finally, $I_\mathrm{nucl}$ takes into account the possible occurrence of nucleation events, corresponding to appearance of pieces of interfaces at some intermediate times. There are a few caveats in the previous statement. As emphasized in \cite{MR}, interfaces need to be counted with their multiplicity, and therefore the natural variable to describe the variational convergence of $(I_\eps)$ is not the order parameter $u$ but rather the general varifold (which does count multiplicity of interfaces) associated to it and the definition of \eqref{i0} has to be extended accordingly. While a $\Gamma$-$\liminf$ estimate for $(I_\eps)$ is proven in \cite{MR}, a corresponding $\Gamma$-$\limsup$ estimate is proven in \cite{KORV} only for special ``nice'' paths. To identify the $\Gamma$-limit it is thus needed a density theorem for the limiting functional $I_0$, which does not appear to be presently available.
 
In the present paper, we fix (suitable) sequences $\lambda_\eps,\gamma_\eps\to 0$ and consider directly the asymptotics of the stochastic equation \eqref{1.3b} for space dimension $d\le 3$. Under natural assumptions on the initial datum, we prove the large deviations upper bound with speed $\eps\lambda_\eps$ and rate function that, in the simplest case of interfaces with multiplicity one, reads,  
\begin{equation}
  \label{i0s}
  I(u) = \frac \tau4 \int_0^T\! \int_{\Sigma_t} \big| \nu_t
  - H_t\big|^2  \, \rmd \mc H^{d-1} \, \rmd t
  +I_\mathrm{sing} (u)\;.
\end{equation}

The rate function here derived improves the one introduced in \cite{MR} in two aspects. We provide a variational characterization of $I_\mathrm{sing}$ in \eqref{i0s} that is strictly larger of $I_\mathrm{nucl}$ in \eqref{i0}. With this characterization, it is readily seen that the zero level set of $I$ is given, as it should be, by the motions by mean curvature in the Brakke formulation. Besides, in describing the large deviations asymptotics, we do not only consider the general varifold associated to $u$, but include the order parameter $u$ itself. We show that the rate function $I$ is finite only if the map $t\mapsto u_t$ is continuos in $L^1$. This exclude the occurrence of spurious nucleation events; essentially, it implies that outside the jump set of $u$ only nucleations with even multiplicity are allowed. This cannot be detected by looking only at the varifold.

From a technical viewpoint, our results will be obtained by suitably blending arguments from the analysis of the action functional, mostly imported from \cite{MR} (which relies on previous results, e.g., \cite{Ilmanen,RS,HT}), with basic tools of stochastic calculus and large deviations estimates for Markov processes. The restriction $d\le 3$ is inherited both from the analysis of the regularity properties of the stochastic equation \eqref{1} \cite{BBP1}, and, as in \cite{MR}, from the validity of the static result in \cite{RS}. 

We remark that, although the model equations are quite different, our analysis has similar features to the one of stochastic conservation laws in \cite{Mariani}.  Finally, we mention that the large deviations asymptotics of a different stochastic perturbation of the Allen-Cahn equation has been recently analyzed in \cite{HR}. 

\section{Notation and results}
\label{sec:2}

We start by introducing, referring to \cite{Simon} for a detailed exposition, the tools from geometric measure theory that are relevant for our purposes. We denote by $\bb T^d$ the $d$-dimensional torus $\bb R^d/\bb Z^d$ and by $\rmd x$ the Haar measure on $\bb T^d$. In the sequel, we systematically identify functions (respectively measures) on $\mathbb{T}^d$ with $1$-periodic functions (respectively measures) on $\bb R^d$. Throughout the paper, we shall shorthand $L^p=L^p(\bb T^d)$, $p\in [1,+\infty]$, and let $H^s=H^s(\bb T^d)$, $s\in \bb R$, be the fractional Sobolev space. Finally, given a topological space $E$ we denote by $C_K(E)$ the set of continuous functions on $E$ with compact support and by $\mc B (E)$ its Borel $\sigma$-algebra. 

\subsection{Rectifiable measures}
We denote by $\mc M$ the set of (signed) Radon measures on $\bb T^d$, and by $\mc M_+$ its positive cone. Furthermore, we let $\mc H^{d-1}$ (respectively $\mc H^{d-1}_{\bb R^d}$) be the $(d-1)$-dimensional Hausdorff measure on $\bb T^d$ (respectively $\bb R^d$).

A set $M\subset\bb T^d$ is \emph{rectifiable} (more precisely ($d-1$)-countably rectifiable) iff there exists a countable collection $(\phi_k)$ of Lipschitz functions from $\bb R^{d-1}$ to $\bb T^d$ such that $\mc H^{d-1}(M\setminus  \bigcup_k \phi_k(\bb R^{d-1}))=0$. 

Given $\mu\in \mc M_+$, the \emph{tangent measure} (more precisely the ($d-1$)-dimensional tangent measure) of $\mu$ at $x\in \bb T^d$ is the positive Radon measure $T_x\mu$ on $\bb R^d$ defined by
\[
T_x\mu(\phi) := \lim_{\lambda \downarrow 0} \frac1{\lambda^{d-1}} \int_{\bb R^d} \phi\circ \eta_{x,\lambda}  \rmd\mu \; , \qquad\phi\in C_K(\bb R^d)\;,
\]
provided the limit exists, where $\eta_{x,\lambda}: \, \bb R^d \to \bb R^d$ is defined by  $\eta_{x,\lambda}(y) = \lambda^{-1}(y-x)$. 

\begin{defin}[Rectifiable and integral measures]
\label{def:0}
A measure $\mu\in \mc M_+$ is called \textit{rectifiable} (more precisely ($d-1$)-\textit{rectifiable}) if either of the following equivalent conditions is met.
\begin{itemize}
\item[a)] $\rmd\mu = \theta\, \rmd \mc H^{d-1}\res M $ for some rectifiable $\mc H^{d-1}$-measurable set $M$ and some $\theta \in L^1(\mc H^{d-1}\res M; (0,\infty))$.
\item[b)] For $\mu$-a.e.\ $x\in \bb T^d$, a tangent measure $T_x\mu$ exists, it is unique, and it is given by
\begin{equation}
\label{Theta}
T_x\mu = \theta(x) \mc H^{d-1}_{\bb R^d}\res \Sigma(x)\;,
\end{equation}
for some $(d-1)$-plane $\Sigma(x)$ of $\bb R^d$ and some strictly positive real 

$\theta \in L^1(\mc H^{d-1}\res M; (0,\infty))$.
\end{itemize}
The ($d-1$)-plane $\Sigma(x)$ in \eqref{Theta} is called the \textit{tangent plane} of $\mu$ at $x$ and will be denoted by $\tau_x\mu$. The real $\theta(x)$ is called the \textit{multiplicity} of $\tau_x\mu$ and will be denoted by $\theta(\mu,x)$. 

\noindent
A rectifiable measure $\mu$ is called \emph{integral} iff $\mu$-a.e.\ the multiplicity is an integer, i.e., $\theta(\mu, \cdot)\in \bb N$.
\end{defin}
We regard $BV(\bb T^d;\{\pm1\})$ as a subset of $L^1$. Given $u\in BV(\bb T^d;\{\pm1\})$, we denote by $S_u$ the so-called measure theoretic boundary of $\{u = 1\}$, i.e., the set of points where $u$ is essentially discontinuous, which is a rectifiable set. Moreover, by denoting with $|\nabla u|$ the total variation measure of $u$, it is a rectifiable integral measure and, more precisely, $|\nabla u| = 2 \mc H^{d-1}\res S_u$. Furthermore, there exists $\bs n\in L^1(|\nabla u|;\bb R^d)$ such that $\rmd\nabla u = \bs n\, \rmd|\nabla u|$ and, for $|\nabla u|$-a.e.\ $x$, $|\bs n(x)|=1$ and $\bs n(x) \perp \tau_x|\nabla u|$.

\subsection{Varifolds}

A \emph{general varifold} (more precisely, a general ($d-1$)-varifold) is a positive Radon measure on $\bb T^d \times \Lambda_{d-1}$, where $\Lambda_{d-1}$ is the Grassmanian manifold of unoriented $(d-1)$-planes in $\bb R^d$. We denote by $\mc V$ the set of all general varifolds. 

A general varifold $V\in \mc V$ can be disintegrated as $V(\rmd x,\rmd\Sigma) = \mu(\rmd x)\,\wp_x(\rmd\Sigma)$, where $\mu \in \mc M_+$ and, for $\mu$-a.e.\  $x\in \bb T^d$, $\wp_x$ is a probability measure on $\Lambda_{d-1}$. The measure $\mu$ is called the \emph{mass measure} of $V$ and will be denoted by $|V|$.	

In the sequel, we shall denote by $a\cdot b$ the inner product between the vectors $a,b\in\bb R^d$, and by $|a|$ the associate Euclidean norm. Given $a\neq 0$ we denote by $a^{\perp}$ the $(d-1)$-plane orthogonal to $a$. For $\Sigma\in \Lambda_{d-1}$, we also denote by $\Sigma$ the orthogonal projection onto $\Sigma$. 

The \emph{first variation} $\delta V$ of $V\in \mc V$ is the linear functional on $C^1(\bb T^d;\bb R^d)$ defined by
\[
\delta V(\eta) = \int\! \mathrm{Tr}(D\eta^\top\Sigma)\,  V(\rmd x,\rmd\Sigma)\;, \qquad \eta\in C^1(\bb T^d;\bb R^d)\;,
\]
where $D\eta$ is the Jacobian matrix of $\eta$ and the superscript $\top$ denotes transposition. If $\delta V$ is a $\bb R^d$-valued Radon measure, absolutely continuous with respect $|V|$, then $\delta V$ can be represented as 
\[
\delta V(\eta) = - |V|(\eta\cdot H)\;,
\]
for some $H\in L^1(\bb T^d,|V|; \bb R^d)$, which is called the (weak) \emph{mean curvature vector}.

A general varifold $V$ is \emph{rectifiable} iff there exists a rectifiable measure $\mu\in \mc M_+$ such that
\[
\int\! f(x,\Sigma) V(\rmd x, \rmd\Sigma) = \int\! f(x,\tau_x\mu) \, \mu(\rmd x)\;, \qquad f\in C(\bb T^d \times \Lambda_{d-1})\;.
\]
Note that if such $\mu$ exists then $\mu = |V|$ and $V(\rmd x,d\Sigma)= \mu(\rmd x) \delta_{\tau_x\mu} (\rmd \Sigma)$. 

Finally, a rectifiable varifold $V\in \mc V$ is called \emph{integral} iff $|V|$ is an integral measure. Observe that there is a one-to-one correspondence between integral varifolds and integral measures.

\smallskip 
Let $\mc M(\bb T^d\times \Lambda_{d-1})$ be the set of Radon measures on $\bb T^d\times \Lambda_{d-1}$ equipped with the total variation norm. Given $T>0$, we denote by $L^\infty([0,T];\mc M(\bb T^d\times \Lambda_{d-1}))$ the set of maps (up to a.e.\ equivalence) $t\mapsto V_t$ essentially bounded and weak*-measurable, i.e., such that $t\to V_t(f)$ is measurable for any $f\in C(\bb T^d\times \Lambda_{d-1})$. Notice that $L^\infty([0,T];\mc M(\bb T^d\times \Lambda_{d-1}))$ is the dual of the separable Banach space $L^1([0,T];C(\bb T^d\times \Lambda_{d-1}))$ \cite{TT}. Thus $L^\infty([0,T];\mc M(\bb T^d\times \Lambda_{d-1}))$ can be endowed with the bounded weak* topology; namely, by definition, a set is open iff its intersection with each bounded set is relatively open in the weak* topology. In addition, norm bounded subsets in $L^\infty([0,T];\mc M(\bb T^d\times \Lambda_{d-1}))$ are metrizable and precompact in the bounded weak* topology. We regard $\mc V$ as a subset of $\mc M(\bb T^d\times \Lambda_{d-1})$ and set,
\begin{equation}
\label{V:=}
\bs V := L^\infty([0,T];\mc V) \quad\text{endowed with the bounded weak* topology}\;, 
\end{equation}
i.e., $\bs V$ is the positive cone of $L^\infty([0,T];\mc M(\bb T^d\times \Lambda_{d-1}))$ endowed with the relative topology. Elements of $\bs V$ are denoted by $V=(V_t)_{t\in [0,T]}$. The following definition of $L^2$-flows has been introduced in \cite{MR}.

\begin{defin}[$L^2$-flows]
\label{def:1}
An element $V\in \bs V$ is called an $L^2$-flow provided it meets the following three conditions.
\begin{itemize}
\item[{a)}] For a.e.\ $t\in [0,T]$, $V_t$ is an integral varifold. 
\item[{b)}] $\displaystyle
\sup_\eta \Big\{ \int_0^T\! \delta V_t(\eta_t) \, \rmd t - \frac 12 \int_0^T\!  |V_t|(|\eta_t|^2) \, \rmd t\Big\} <\infty\;,
$
where the supremum is carried out over $\eta\in C^1([0,T]\times \bb T^d;\bb R^d)$. 

\item[{c)}] There exists $\nu \in L^2([0,T]\times\bb T^d, |V_t|\,\rmd t;\bb R^d)$ such that
\begin{equation}
\label{nut1}
\nu_t(x) \perp \tau_x|V_t| \qquad |V_t|\, \rmd t\text{ - a.e.} 
\end{equation}
and 
\begin{equation}
\label{nut2}
\sup_\psi
\: \int_0^T\, |V_t|\big(\partial_t \psi_t 
+  \nabla \psi_t \cdot \nu_t \big) \, \rmd t 
< + \infty\;,
\end{equation}
where the supremum is carried over all $\psi\in C^1_K((0,T)\times \bb T^d)$ such that $\|\psi\|_\infty\le 1$. 
\end{itemize} 
\end{defin}
By Riesz's representation lemma, b) implies that, for a.e.\ $t\in [0,T]$, $V_t$ admits a mean curvature vector $H_t$ and $(H_t)_{t\in [0,T]}$ belongs to $L^2([0,T]\times\bb T^d, |V_t|\,\rmd t;\bb R^d)$. Any vector $\nu \in L^2([0,T]\times\bb T^d, |V_t|\,\rmd t;\bb R^d)$ satisfying condition c) is called a \emph{velocity} of the $L^2$-flow $V$. As proven in \cite[Prop.~3.3]{MR}, $\nu$ is uniquely determined in the points $(t,x)\in (0,T)\times \bb T^d$ where both tangential planes $T_{(t,x)}|V|$ and $T_x|V_t|$ exist. However, it is not known whether this uniqueness set has full $|V_t|\rmd t$-measure. 

\begin{rem}
\label{rem:otto}
If $V$ is an $L^2$-flow then the map $ (0,T)\ni t\mapsto |V_t|(\phi)$ has bounded variation for each $\phi\in C^1(\bb T^d)$ and therefore, as observed in \cite[Rem.~3.2]{MR}, it is possible to choose a representative for which there exists a countable set $D_V\subset (0,T)$ such that the map $t\mapsto |V_t|(\phi)$ is continuous on $(0,T)\setminus D_V$ for any $\phi\in C^1(\bb T^d)$. Furthermore, in view of the mass bound $\mathrm{ess} \sup_{0<t<T} \| |V_t| \|_{TV} <\infty$, it is easy to construct a function $\mu:[0,T] \to \mc M_+$ such that $\mu_t= |V_t|$ for $t \in [0,T] \setminus D_V$ and $t \mapsto \mu_t(\phi)$ is c\`adl\`ag (or c\`agl\`ad), i.e., right-continuous with left limits, for every $\phi\in C(\bb T^d)$.

\end{rem}

\subsection{The model}
We consider the Allen-Cahn equation on the $d$-dimensional torus $\bb T^d$, $d\le 3$, with scaling parameter $\eps$ and double well potential $W$, stochastically perturbed by a space-colored noise that however becomes white in the limit $\eps\to 0$. 

The assumptions on the potential $W$, which have been tailored to include the paradigmatic case $W(u)=\frac14 (1-u^2)^2$, are detailed below.

\begin{assumption}[Assumptions on $W$]
\label{t:ws}
\begin{enumerate} 
\item $W\in C^2\big(\bb R ;[0,+\infty)\big)$, $W(u)=0$ iff $u=\pm 1$,
  $W''(\pm 1)>0$, and $W$ is uniformly convex at infinity, i.e.,
  there exists a constant $C\in (0,+\infty)$ and a compact $K\subset
  \bb R$ such that $W''(u) \ge \frac{1}{C}$ for any $u\not\in K$.
\item $W$ has at most growth $4$, i.e., there
  exists a constant $C\in (0,+\infty)$ such that $|W(u)|\le C
  (|u|^4+1)$ for any $u\in \bb R$. 
\item $W'$ has at most growth $3$, i.e., there
  exists a constant $C\in (0,+\infty)$ such that $|W'(u)|\le C
  (|u|^{3}+1)$ for any $u\in \bb R$. 
\item There
  exists a constant $C\in (0,+\infty)$ such that 
  $|W''(u)|\le C (\sqrt{W(u)}+1)$ for any $u\in \bb R$. 
\end{enumerate}
\end{assumption}

Hereafter, $u_\pm =\pm 1$ are the pure phases and $\tau = \int_{-1}^1\! \sqrt{2W(s)} \, \rmd s$ is the surface tension with $W$ satisfying the above assumptions.

The dynamics is specified by the stochastic partial differential equation,
\begin{equation}
  \label{1}
  \rmd u = \left[\Delta u - \frac1{\eps^2} W'(u) \right]\rmd t
  +  \sqrt{2\lambda_\eps}\, \rmd \alpha^\eps_t\;, 
\end{equation}
where $\lambda_\eps>0$ and $\alpha^\eps$ is the Gaussian process on $C([0,T];H^{-s})$, $s>d/2$, with mean zero and covariance,
\[
\bb E \big[\alpha^\eps_t(\varphi)\,\alpha^\eps_{t'}(\psi)\big] = t\wedge t'\,
\big\langle j_\eps*\varphi\;,\, j_\eps*\psi \big\rangle_{L^2}\;, \qquad
\varphi,\psi\in H^s\;,
\]
in which $j_\eps\in H^1$ is an approximation to the Dirac $\delta$, and $*$ denotes the convolution on $\bb T^d$. 

Given $T>0$, $\eps>0$, and $\bar u_0^\eps\in H^1$, as proven in \cite{BBP1}, there exists a unique process in $C([0,T];L^2)$ that solves the Cauchy problem for \eqref{1} with initial condition $\bar u_0$. We denote by $\bb P_\eps$ the law of this solution that, again by \cite{BBP1}, satisfies $\bb P_\eps (u \in C([0,T];H^1)\cap L^2([0,T];H^2))=1$. 
The main aim is to analyze the asymptotic behavior of \eqref{1} in the singular limit $\eps\to 0$ and $\lambda_\eps\to 0$. To carry out this analysis the following condition on $j_\eps$ and $\lambda_\eps$,
\begin{equation}
\label{2.5b}
\lim_{\eps\to 0} \left(\eps\lambda_\eps\big\|\nabla j_\eps\big\|_{L^2}^2 +
  \eps^{-1}\lambda_\eps \big\|j_\eps\big\|_{L^2}^2\right) = 0\;,
\end{equation}
is enforced through the paper. Notice, for instance, that if $j_\eps(\cdot) = \eps^{-\beta d}j(\cdot/\eps^\beta)$, $0<\beta\le 1$, for some $j \in H^1$, then \eqref{2.5b} holds when $\lambda_\eps = o(\eps^{1+ \beta d})$. 

The deterministic Allen-Cahn equation, i.e., \eqref{1} without noise, is the $L^2$-gra\-dient flow of the van der Waals' free energy functional $\mc F_\eps\colon L^1 \to [0,+\infty]$ defined by 
\begin{equation}
\label{F}
\mc F_\eps(u) := \begin{cases} {\displaystyle \int\!\Big[ \frac\eps2 |\nabla u|^2 + \frac 1\eps W(u) \Big]\,\rmd x} & \text{ if $u\in H^1$}, \\ \; +\infty & \textrm{ otherwise}\;. \end{cases}
\end{equation}
Observe that since $W$ has at most quartic growth and $d\le 3$, by Sobolev embedding, $u\in H^1$ implies $W(u)\in L^1$. 

Given $u\in L^2$, we introduce the \emph{free energy measure}, as the positive Radon measure on $\bb T^d$ defined by 
\begin{equation}
  \label{en}
  \mu_\eps^u(\rmd x) := 
\begin{cases}  
  {\displaystyle \Big[\frac \eps 2 |\nabla u|^2 + \frac 1\eps
  W(u)\Big] \rmd x} & \textrm{if } u\in H^1\;, \\
  \; 0  &  \textrm{otherwise}\;, 
\end{cases}
\end{equation}
and the associate general varifold
\begin{equation}
  \label{va}
  V_\eps^u(\rmd x,\rmd \Sigma) :=
   \begin{cases}
   \mu_\eps^u(\rmd x)\,
  \delta_{(\mathbf{n}^u)^\perp}(\rmd\Sigma) & \textrm{if } u\in H^1 \;, \\
\; 0  & \rm{otherwise}\,.
   \end{cases}
   \end{equation}
Here, the unit vector $\mathbf{n}^u$ is given by 
  \begin{equation}
  \label{nu=}
    \mathbf{n}^u:= 
   \begin{cases}
   {\displaystyle \frac{\nabla u}{|\nabla u|}\;, } & \textrm{if } \nabla u \neq 0\;, \\
  \; e_0 &  \textrm{otherwise}\;,  
\end{cases}
\end{equation}
where, for $u\in H^1$, the vector $\nabla u$ is defined $\rmd x$-a.e.\ and $e_0$ is an arbitrary fixed unit vector. In particular, $|V_\eps^u| = \mu_\eps^u$.

The initial datum  $\bar u_0^\eps$ is assumed deterministic and meeting the following conditions.

\begin{assumption}[Conditions on the initial datum] $~$
  \label{t:au0}
  \begin{itemize}
  \item[a)] $(\bar u^\eps_0)_{\eps>0}\subset H^1$ and 
    $\displaystyle{\varlimsup_{\eps\to 0} 
      \mc F_\eps(\bar u^\eps_0) <+\infty.}$ 
  \item[b)]  As $\eps\to 0$ the sequence $(\bar u^\eps_0)$
    converges in $L^1$ to some $\bar{u}_0\in BV(\bb T^d; \{\pm 1\})$. 
  \item[c)] 
    As $\eps\to 0$ the sequence $(\mu_\eps^{\bar{u}^\eps_0})$
    converges as Radon measure to some $\bar{\mu}_0$. 
  \end{itemize}
\end{assumption}

Observe that the requirement of the convergences in items b) and c) follows, possibly by extracting a subsequence, from the equi-boundedness in a), see, e.g., \cite{Modica}. 

Our aim is to investigate the asymptotic behavior of the sequence of probabilities $(\bb P_\eps)_{\eps>0}$ as $\eps\to 0$. To this end, set 
\begin{equation}
\label{U:=}
\bs U := C([0,T];L^1) \quad \text{endowed with the norm topology}\;,
\end{equation}
recall the definition \eqref{va} of the general varifold associated to a profile $u\in L^2$, and the definition of the space $\bs V$ in \eqref{V:=}. Given $u=(u_t)_{t\in [0,T]}\in C([0,T];L^2)$, we let $V_\eps^u\in \bs V$ be defined by $V_{\eps,t}^{u} = V_\eps^{u_t}$, $t\in[0,T]$, if $u\in C([0,T];H^1)$, and $V_{\eps,t}^{u} = 0$, $t\in[0,T]$, otherwise. Since $\bs V$ is endowed with the bounded weak* topology, by Lemma \ref{mesura} the map $C([0,T];L^2) \ni u \mapsto V_\eps^u \in \bs V$ is Borel measurable and therefore the map $C([0,T];L^2) \ni u \mapsto (u,V_\eps^u) \in \bs U \times \bs V$ is $\mc B(\bs U)\otimes \mc B(\bs V)$ measurable. Note that since $\bs U$ has a countable basis then $\mc B(\bs U \times \bs V) = \mc B(\bs U)\otimes \mc B(\bs V)$, see \cite[Lemma 6.4.2]{Bogachev}. We can thus regard $\big(\bb P_\eps \circ (u,V_\eps^u)^{-1}\big)_{\eps>0}$ as a sequence of probabilities on $(\bs U \times \bs V, \mc B(\bs U \times \bs V))$ and analyze its large deviations asymptotics as $\eps \to 0$. To formulate such large deviations principle, we need however a few more notation and definitions. 

\subsection{Admissible pairs}

It turns out that not all the elements in $\bs U\times \bs V$ are significant for the large deviations asymptotics and here we describe the relevant ones as cluster points of $\big((u^\eps, V_\eps^{u^\eps})\big)_{\eps>0}$ for suitable (deterministic) sequences $(u^\eps)_{\eps>0}$.  Unfortunately, this description is somewhat technically involved; it has been engineered to make the rate function of the large deviations upper bound (that we prove) as large as possible, and to guarantee its goodness (i.e., its coercivity and lower semicontinuity). The proof of a matching lower bound (that we do not discuss) should rely on a suitable density theorem. In this respect, the characterization of the rate function here provided might be of some help.

Given $u\in\bs{U}$ and $\delta>0$ we denote by $\omega^\infty(u;\delta)$ its continuity modulus, i.e.,
\begin{equation}
  \label{ominf}
  \omega^\infty(u;\delta) := 
  \sup_{\substack{t,s\in[0,T]\\ |t-s|\le \delta}}
  \| u_t -u_s\|_{L^1}\;.
\end{equation}
Given $z\in L^1([0,T])$, according to the Kolmogorov-Riesz-Fr\'echet compactness criterion (see, e.g., \cite[Thm.~4.26]{Brezis}), we let $\omega^1(z;\delta)$ be its $L^1$-continuity modulus regarding $L^1([0,T])$ as a subset of $L^1(\bb R)$, i.e.,
\begin{equation}
  \label{om1}
  \omega^1(z;\delta) := \sup_{\delta'\in(0,\delta]}
  \Big(\int_0^{\delta'}\! (|z_t|+|z_{T-t}|) \, \rmd t + \int_{\delta'}^T\!
  |z_t-z_{t-\delta'}|\, \rmd t \Big)\;.  
\end{equation}
Finally, we introduce the diffuse Willmore functional $\mc W_\eps\colon L^1 \to [0,+\infty]$, defined by 
\begin{equation}
\label{W}
\mc W_\eps(u) := \begin{cases} {\displaystyle \frac 1 \eps \int\! \Big(\eps \Delta u - \frac 1\eps W'(u) \Big)^2\,\rmd x} & \text{ if $u\in H^2$}. \\ +\infty & \text{ otherwise}. \end{cases}
\end{equation}
Observe that since $W'$ has at most cubic growth and $d\le 3$, by Sobolev embedding, $u\in H^1$ implies $W'(u)\in L^2$. 

In the following definition we fix $\alpha_2\in (0,\frac{1}{4d})$, $\alpha_3 \in (0,\frac 12)$, and a countable set $\{ \phi_j \}\subset C^1(\bb T^d)$, dense in the unit ball. The condition $\alpha_2 < \frac{1}{4d}$ is not optimal and due to technical issues.
\begin{defin}[Admissible pairs]
  \label{def:2}
  Recall the definitions of $\bs V$ and $\bs U$ in \eqref{V:=} and \eqref{U:=}.
  Given $\bs{\ell}=(\ell_1,\ell_2,\ell_3)\in \bb (0,\infty)^3$, let
  $\bs\Gamma_\bs{\ell}$, the set of \emph{$\bs\ell$-admissible
  pairs}, be the collection of elements in $\bs{U}\times \bs {V}$
  such that $(u,V) =\lim_\eps (u^\eps, V_\eps^{u^\eps})$ in the topology of $\bs{U}\times \bs {V}$ for some sequence $(u^\eps)_{\eps>0}\subset C([0,T];H^1)\cap L^2([0,T];H^2)$, $\eps \downarrow 0$, meeting the following
  conditions for any $\eps$ and for any $\delta \in (0,T]$,
  \begin{itemize}
  \item[a)]
    $u^\eps_0=\bar{u}^\eps_0$ with $(\bar{u}^\eps_0)_{\eps>0}$ as in
    Assumption~\ref{t:au0}.     
  \item[b)] 
    $\displaystyle{ 
      \sup_{t\in[0,T]} \mc F_\eps(u^\eps_t) + \int_0^T \! \mc
      W_\eps(u^\eps_t) \, dt  \le \ell_1 .}  $
  \item[c)] 
    $\displaystyle{
         \omega^\infty(u^\eps;\delta) \le
      \ell_2 \delta^{\alpha_2} \, .
    }$
  \item[d)] Letting $z^\eps(\phi)\in L^1([0,T])$ be defined by $z^\eps(\phi)_t := |V_{\eps,t}^{u^\eps}|(\phi)$, $\phi\in C^1(\bb T^d)$, then for any $1 \leq j \leq  \lfloor (\eps \lambda_\eps)^{-1} \rfloor$ we have $\displaystyle{
      \omega^1(z^\eps(\phi_j);\delta) 
      \le \|\phi_j\|_{C^1(\bb T^d)}\, \ell_3 \delta^{\alpha_3}}.$
      
 \end{itemize}
We also define $\bs\Gamma:=\bigcup_{\bs\ell} \bs\Gamma_{\bs\ell}$ that will be called the set of \emph{admissible pairs}. An element $V\in\bs V$ is called \emph{admissible} iff $(u,V)$ is an admissible pair for some $u\in\bs U$.
\end{defin} 

The next statement, which relies on results in \cite{Modica,RS}, as detailed in Appendix~\ref{app:b}, shows that in dimension $d\le 3$ the admissible pairs enjoy nice properties.

\begin{theorem} 
  \label{t:ap} 
Recall $\tau$ denotes the surface tension as defined in \eqref{tau}.  
For each $\bs\ell\in\bb (0, +\infty)^3$ the set $\bs\Gamma_{\bs\ell}$ is compact in $\bs U\times \bs V$. Furthermore, if $(u,V)\in \bs\Gamma_{\bs\ell}$ then $u_0=\bar{u}_0$ as in Assumption \ref{t:au0} and for any $\delta \in (0,T]$,
  \begin{itemize}
  \item[a)] $u\in L^\infty\big([0,T]; BV(\bb T^d; \{\pm 1\})\big)$, 
    $\essup_{t}\|u_t\|_\mathrm{TV}\le 2\ell_1/\tau$, 
        and \par
        $\displaystyle{
        \omega^\infty(u;\delta) \le
      \ell_2 \delta^{\alpha_2}\, ;
    }$
  \item[b)] $\essup_{t}\||V_t|\|_\mathrm{TV}\le \ell_1$
    and 
    \begin{itemize}
    \item[b.1)] for a.e.\ $t\in [0,T]$, $\tau^{-1}V_t$ is an integral
      varifold,
    \item[b.2)] for a.e.\ $t\in [0,T]$, $V_t$ admits a mean
      curvature $H_t$ which satisfies 
      
      $\int_0^T\!  |V_t|(|H_t|^2)\, \rmd t \le \ell_1$,
    \item[b.3)] for any $\phi \in C^1(\bb T^d)$ it holds $\displaystyle{
       \omega^1(z(\phi);\delta) 
      \le \|\phi\|_{C^1(\bb T^d)}\, \ell_3 \delta^{\alpha_3}}\;,$ where $z(\phi)_t := |V_t|(\phi)$;
    \end{itemize}
  \item[c)] for a.e.\ $t\in [0,T]$, $\frac 12 \, \rmd |\nabla u_t| \le \frac 1\tau \, \rmd |V_t|$.
\end{itemize} 
\end{theorem}
Statement c) could be improved. Indeed, arguing as in \cite[Thm.~1]{HT}, it should be actually possible to show that for a.e.\ $t\in [0,T]$ one has $|V_t|=\frac{\tau}2|\nabla u_t|+ \tau \tilde{\mu}_t$, where $\tilde{\mu}_t$ is a rectifiable measure with even multiplicity.

\subsection{Brakke motion}

For the present purpose of describing the asymptotic behavior of the stochastically perturbed Allen-Cahn equation, we adopt a slightly different definition of (weak) motion by mean curvature with respect to the one of Brakke motion \cite{Brakke}.
 
\begin{defin}[Brakke motion]
  \label{t:bm}
  Given a Radon measure $\bar{\mu}_0 \in \mc M_+$, an element $V\in
  \bs V$ is called a \emph{Brakke motion} with initial datum
  $\bar \mu_0$ iff $V$ is admissible and for each $\psi \in C^1_K\big([0,T)\times \bb T^d;\bb R_+\big)$,
    \begin{equation}
    \label{inbra}
     -\bar \mu_0(\psi_0) +\int_0^T\! 
     |V_t|\big( -\partial_t \psi_t - H_t\cdot \nabla\psi_t 
     + |H_t|^2 \psi_t\big) \, \rmd t \le 0\;,
    \end{equation}
    where $H_t$ is the mean curvature vector of $V_t$. 
    Given $\bar u_0\in BV(\bb T^d;\{\pm 1\})$ and $\bar{\mu}_0 \in \mc M_+$, a pair $(u,V)\in \bs U\times \bs V$ is
    called an \emph{enhanced Brakke motion} with initial
    datum $(\bar u_0,\bar \mu_0)$ iff $(u,V)$ is an
    admissible pair, $V$ is a Brakke motion with
    initial datum $\bar \mu_0$, and 
    $u_0 = \bar u_0$ (compare with \cite[Sect.~12]{Ilmanen} and \cite{Ilmanen0}).  
\end{defin}

In view of Theorem \ref{t:ap}, if $V$ is admissible then it admits a mean curvature vector in $L^2([0,T]\times\bb T^d, |V_t|\,\rmd t;\bb R^d)$, which implies that the above definition is well posed. Moreover, if $V$ is a Brakke motion then $\tau^{-1}V$ is an $L^2$-flow with velocity $\nu =H$. Indeed, for a suitable choice of the positive test function, the inequality \eqref{inbra} easily implies \eqref{nut2}, while the orthogonality condition \eqref{nut1} follows from orthogonality of the mean curvature vector for integral varifolds \cite[Chap.~5, pag.~121]{Brakke}.

It is possible to show that the previous definition of Brakke motion implies the usual one. More precisely, if $V$ is a Brakke motion with initial datum $\bar \mu_0$ and $\mu_t$ is the c\`agl\`ad representative of $|V_t|$, $t\in [0,T]$, introduced in Remark \ref{rem:otto}, then for each $\phi\in C^1(\bb T^d)$ and each $t\in[0,T)$,
  \begin{equation*}
    \varlimsup_{s\to t} \frac{\mu_s(\phi) -\mu_t(\phi)}{s-t} 
    \le \mu_t\big( H_t\cdot \nabla\phi  - |H_t|^2 \phi \big)\;,
  \end{equation*}
where we understand that the right-hand side is $-\infty$ for the (zero measure) set of times such that either $H_t$ does not exist or does not belong to $L^2(\mu_t;\bb R^d)$, see \cite[Thm.~7.1]{ES}. Furthermore, \eqref{inbra} implies, in consistence with the possible instantaneous disappearance of mass, the inequality $\mu_0 \leq \bar{\mu}_0$ as Radon measures. 
 
\subsection{The rate function}

If $V\in \bs V$ is admissible, $\tau^{-1}V$ is an $L^2$-flow, and $\nu$ is a velocity of $\tau^{-1}V$, we set, 
\begin{eqnarray}
  \label{Iac}
&&  \!\!\!\!\!\!\!\! I_\mathrm{ac}(V,\nu) := \frac14
  \int_0^T \! |V_t|\big( \big|\nu_t - H_t\big|^2\big)\, \rmd t
\\ 
  \label{Ising}
&&   \!\!\!\!\!\!\!\! I_\mathrm{sing}(V,\nu) := 
  \sup_\psi \Big\{ - \bar \mu_0(\psi_0) +
  \int_0^T\!|V_t|\big(-\partial_t\psi_t -\nu_t \cdot \nabla\psi_t +
  \nu_t \cdot H_t\, \psi_t \big)\, \rmd t\Big\}\;,\qquad\;
\end{eqnarray}
where the supremum is carried out over all $\psi\in C^1_K([0,T)\times\bb T^d)$ such that $0\le\psi\le 1$.

Recall that $\bs\Gamma$ denotes the set of admissible pairs, see Definition \ref{def:2}, and let $I\colon \bs U\times \bs V \to [0,+\infty]$ be the functional defined by
\begin{equation}
  \label{I=}
  I (u,V) :=
  \begin{cases}
    {\displaystyle \inf_\nu} \{I_\mathrm{ac}(V,\nu) + I_\mathrm{sing}(V,\nu)\} & \textrm{if $(u,V)\in
      \bs\Gamma$, $\tau^{-1}V$ is an $L^2$-flow,} \\
      + \infty & \textrm{otherwise,}
  \end{cases}
\end{equation}
where the infimum is taken over all the possible velocities of $V$.

It is simple to check that $I(u,V)=0$ iff $(u,V)$ is an enhanced Brakke motion with initial datum $(\bar u_0,\bar \mu_0)$ according to Definition~\ref{t:bm}. 

Let $t\mapsto\mu_t$ be the c\`adl\`ag representative of $t\mapsto|V_t|$ introduced in Remark \ref{rem:otto} and denote by $D_V$ its jump set. By localizing the test function $\psi$ in the variational definition \eqref{Ising} around the set $D_V$, we deduce that
\begin{equation}
  \label{mt}
  I_\mathrm{sing}(V,\nu) \ge 
  \sup_\phi \Big(\mu_0 (\phi) - \bar\mu_0(\phi) \Big) 
  + \sum_{t\in D_V } \sup_\phi
  \Big(\mu_t (\phi) - \mu_{t^-} (\phi) \Big) =: I_{\rm nucl}(V)\;,
\end{equation}
where the suprema are carried out over all $\phi\in C^1(\bb T^d)$ such that $0\le \phi\le 1$. The right-hand side in \eqref{mt} is the nucleation part of the rate function introduced in \cite{KORV,MR}. The inequality $I_\mathrm{sing}\ge I_{\rm nucl}$ is strict. Consider indeed an element $V$ such the map $t\mapsto\mu_t$ has no jumps, but $t\mapsto\mu_t(\phi)$ has a derivative  with nontrivial positive Cantor part for some $\phi\in C^1(\bb T^d)$, then $I_\mathrm{sing}(V)>0$ while the right-hand side of \eqref{mt} vanishes. To our knowledge, the possible occurrence of paths $t\mapsto\mu_t$ such that $\mu_t(\phi)$ has a Cantor part for some $\phi\in C^1(\bb T^d)$ cannot be ruled out even in the context of the derivation of Brakke motion as singular limit of the deterministic Allen-Cahn equation. It would be interesting to establish a connection between the rate function \eqref{I=} and the one recently introduced in \cite{MR1}, that is defined by a somewhat analogous variational expression.

\subsection{Large deviations upper bound}

For the reader convenience, we first recall the large deviations axiomatic, see, e.g., \cite{DZ}. Let $(\mc P_\eps)$ be a family of probability measures on a Hausdorff topological space $\mc X$. The family $(\mc P_\eps)$ satisfies the \emph{good large deviations principle} with speed $\beta_\eps\downarrow 0$ and rate function $J\colon \mc X\to [0,+\infty]$ iff the following conditions are met.
\begin{itemize}
\item[i)]({\sl Goodness}) $J$ has compact sub-level sets. 
\item[ii)]({\sl LD upper bound}) For each closed set $C\subset \mc X$,
$\displaystyle
\varlimsup_{\eps\to 0} \beta_\eps \log\mc P_\eps(C) \le - \inf_{C}J
$.
\item[iii)]({\sl LD lower bound}) For each open set $A\subset \mc X$,
$\displaystyle
\varliminf_{\eps\to 0} \beta_\eps \log\mc P_\eps(A) \ge - \inf_{A}J
$.
\end{itemize}
The large deviations estimates ii) and iii) give a precise sense to the (logarithmic) asymptotics $\mc P_\eps(B) \asymp \exp\{-\beta_\eps^{-1} \inf_B J\}$. Observe that if the zero level set of $J$ is the singleton $\{s_0\}$ then the large deviations upper bound together with the goodness of $J$ imply the law of large numbers $\mc P_\eps\to \delta_{s_0}$ (weakly as probability measure), together with an exponential control on the error.

In the setting of the stochastic Allen-Cahn approximation to the mean curvature flow, the following theorem provides a large deviations upper bound.

\begin{theorem}[LD upper bound]
  \label{thm:1}
  Let $d\le 3$ and $\bb P_\eps$ be the law of the solution to
  \eqref{1} with initial condition $\bar u_0^\eps$. The sequence of
  probabilities $\big(\bb P_\eps \circ (u,V_\eps^u)^{-1}\big)$ on
  $\bs U \times \bs V$ satisfies a large deviations
  upper bound with speed $\eps\lambda_\eps$ and good rate function
  $I\colon \bs U \times \bs V\to
  [0,+\infty]$ given by \eqref{I=}. Namely, $I$ has compact sub-level sets and, for each
  closed set $\bs C\subset \bs U \times \bs V$,
  \begin{equation}
  \label{rigetto}
    \varlimsup_{\eps\to 0} \eps\lambda_\eps
    \log\bb P_\eps\big\{(u,V_\eps^u)\in \bs C\big\} \le -\inf_{\bs C} I\;.
  \end{equation}
\end{theorem}

As a corollary of this result, we deduce that the cluster points of the sequence $\big(\bb P_\eps \circ (u,V_\eps^u)^{-1}\big)$ are supported by the enhanced Brakke motions with initial datum $(\bar u_0,\bar \mu_0)$, in the sense of Definition~\ref{t:bm}. Even in the two-dimensional case there are well-known examples in which uniqueness for Brakke mean curvature flow fails, see, e.g., \cite{Brakke,Ilmanen2}. We have thus not obtained a genuine law of large numbers for the stochastically perturbed Allen-Cahn equation. The reasonable hope, but apparently quite impervious to pursue, is that the stochastic perturbation selects the physical motions. In this respect, Theorem~\ref{thm:1} shows that the set of all possible Brakke motions can be achieved with probability not exponentially small, but gives no further informations about the limiting probability laws on this set.

\subsection{Discrepancy measure}
\label{sec:3.2b}

A crucial technical ingredient in the Allen-Cahn approximation of the mean curvature flow is the \textit{limiting equipartition of energy}. For later use, we recall here the precise statement. Given $u\in L^2$ we introduce the \emph{discrepancy measure} as the signed Radon measure on $\bb T^d$ defined by 
\begin{equation}
\label{dm}
\rmd \xi^u_\eps := \begin{cases} \Big(\frac \eps 2 |\nabla u|^2 - \frac 1\eps W(u)\Big) \rmd x\;,  & \textrm{if } u\in H^1 \;, \\
\; 0  & \rm{otherwise}\,.
\end{cases}
\end{equation}
Given $u\in C([0,T];L^2)$ we let $\xi^u_\eps\in L^\infty([0,T];\mc M(\bb T^d))$ be defined by $\xi^u_{\eps,t}=\xi^{u_t}_\eps$ if $u\in C([0,T];H^1)$ and $\xi^u_{\eps,t}= 0$ otherwise. We observe that, as it is well known, the so-called monotone one dimensional entire stationary solutions of the deterministic Allen-Cahn equation satisfy the equipartition of energy $\frac\eps 2|\nabla u|^2 = \frac 1\eps W(u)$. The discrepancy measure quantifies the violation of this equipartition property.  

The following statement is the content of \cite[Prop.~6.1]{MR}.
\begin{lem}
\label{prop:2}
Fix $\ell_1>0$ and let $(u^\eps)$ be a sequence meeting condition \emph{b)} in Definition \ref{def:2}. Then,
\[
\lim_{\eps\to 0} \int_0^T \!\big\| \xi^{u^\eps}_{\eps,t} \big\|_\mathrm{TV} \, \rmd t =0\;.
\] 
\end{lem}

\subsection{Stochastic currents}
\label{sec:3.3}

The definition of curvature has been given for general varifolds and reduces to the classical one when the varifold is rectifiable, its multiplicity is constant, and it is  supported by a smooth surface of codimension one. On the other hand, given $V=(V_t)_{t\in [0,T]}\in \bs V$, its associated velocities are defined only if $\tau^{-1}V$ is an $L^2$-flow, in particular only if $V_t$ is rectifiable for a.e.\ $t\in [0,T]$. Therefore, for $\eps>0$, the velocity of the path $(V_{\eps,t}^u)_{t\in [0,T]}$ has been not defined yet. A similar issue is also present in \cite{MR}, where the velocity for $\eps>0$ is defined to be proportional to $-\eps \nabla u_t\, \partial_t u_t$. By using the measure-function pairs theory developed in \cite{H}, which requires an $L^2$-estimate on $\partial_t u$, in \cite{MR} it is then shown that the limit of such velocities exists in a suitable sense and converges to a velocity of the limiting $L^2$-flow. 

In the present stochastic case, the above strategy is not directly applicable, due to the lack of control of the time derivative of the process. For $\eps>0$, we next define, with $\bb P_\eps$-probability one, the velocity of the general varifold $V_\eps^u$ as a stochastic current, and regard it as a separate variable in the large deviations principle. Relying on suitable super-exponential bounds, at the end of the argument, we are able to show that currents can be represented in terms of velocities of $L^2$-flows.

The stochastic current is defined as follows. Given $s\in\bb R$, let $H^s(\Lambda_{d-1};\bb R^d)$ be the vector-valued fractional Sobolev space on $\Lambda_{d-1}$. For $s\ge 0$, it can be defined as the domain of $(I-\Delta)^{s/2}$ on $L^2(\Lambda_{d-1};\bb R^d)$ equipped with the graph-norm. Here $\Delta$ denotes the Laplace-Beltrami operator on $\Lambda_{d-1}$ endowed with the standard Riemaniann metric. As usual $H^{-s}(\Lambda_{d-1};\bb R^d)$, $s>0$, is defined as the dual of $H^s(\Lambda_{d-1};\bb R^d)$. Observe that if $s>\frac{d-1}2$ then $H^s(\Lambda_{d-1};\bb R^d)\hookrightarrow C(\Lambda_{d-1};\bb R^d)$. Given $s\in \bb R$ and an Hilbert space $H$, we denote by $H^{s}(\bb T^d;H)$ the $H$-valued fractional Sobolev space on $\bb T^d$. It can be defined in terms of the $H$-valued Fourier series on $\bb T^d$ with the usual norm. For $s>\frac d2$ we have $H^s(\bb T^d;H)\hookrightarrow C(\bb T^d;H)$. Given $s\in (-1,1)$ and an Hilbert space $H$, we also let $H^s([0,T]; H)$ be the $H$-valued fractional Sobolev space on $[0,T]$. For $s\in [0,1)$, it can be defined via the standard Gagliardo norm, while, as usual, $H^{-s}([0,T];H)$, $s\in (-1,0)$, is defined, letting $H'$ be the dual of $H$, as the dual of $H^{s}([0,T];H')$. For $s>\frac 12$ we have $H^{s}([0,T];H)\hookrightarrow C([0,T];H)$. Finally, given $\bs s = (s_1,s_2,s_3)\in (-1,1)\times \bb R^2$, we set 
\begin{equation}
\label{Hs:=}
\bs H^{\bs s} := H^{s_1}([0,T];H^{s_2}(\bb T^d;H^{s_3}(\Lambda_{d-1};\bb R^d)))\;.
\end{equation}  
Observe that if $\bs s\in (\frac 12,1)\times (\frac d2,\infty)\times (\frac{d-1}2,\infty)$ then $\bs H^{\bs s}\hookrightarrow C([0,T]\times \bb T^d\times \Lambda_{d-1};\bb R^d)$.

For $\bs s\in (\frac 12,1)\times (\frac d2 ,+\infty)\times (\frac{d-1}2,+\infty)$ and $f\in \bs H^{\bs s}$ we define,
\begin{equation}
\label{J}
J_\eps^u(f) := - \eps\int_0^T\! \langle \nabla u_t\cdot f_t(\cdot,(\mathbf{n}^u)^\perp),\rmd u_t \rangle_{L^2}\;,
\end{equation}
where we recall that $\mathbf{n}^u$ has been defined in \eqref{nu=} and the right-hand side is $\bb P_\eps$-a.s.\ defined as an It\^o's stochastic integral with respect to the semimartingale $u$ (for the latter notions see, e.g., \cite[Chap.~4]{DZ}). As follows from the theory of stochastic currents for \eqref{1} developed in Appendix \ref{app:c} (analogous to the analysis of stochastic currents for finite dimensional diffusions in \cite{FGGT}), the map $f\mapsto J_\eps^u(f)$ defines, with $\bb P_\eps$-probability one, a linear functional on $\bs H^{\bs s}$. We shall denote by $u\mapsto J_\eps^u$ the associated $\bs H^{-\bs s}$-valued random variable.

\begin{rem}
\label{rem:2}
Let $\bs L$ be the closure of the linear subspace of $\bs H^{\bs s}$ of functions of type $f_t(x,\Sigma) = \Sigma \, \eta_t(x)$, $\eta\in H^{s_1}([0,T]; H^{s_2}(\bb T^d;\bb R^d))$ (recall that for $\Sigma\in\Lambda_{d-1}$, the orthogonal projection onto $\Sigma$ is still denoted by $\Sigma$). From the very definition of $J_\eps^u$, it vanishes on $\bs L$.
\end{rem}

In the sequel, we shall regard $\eta\in H^{s_1}([0,T]; H^{s_2}(\bb T^d;\bb R^d))$ also as the element in $\bs H^{\bs s}$ defined by $f^\eta_t(x,\Sigma) = \eta_t(x)$, and we shorthand $J_\eps^u(f^\eta)$ by $J_\eps^u(\eta)$.

\section{Super-exponential estimates}
\label{sec:3}

In this section we prove the probability estimates needed for the large deviations upper bound. These will be achieved by suitable applications of It\^o's formula with respect to various semimartingales whose quadratic variations will be explicitly computed (for an introduction to these notions we refer the unfamiliar reader to, e.g., \cite[Chap.~4]{DZ}). Strictly speaking, It\^o's formula will be applied to some functions that are not $C^2$. Nevertheless, the resulting formulae can be justified by means of an appropriate truncation procedure, that is here completely omitted and not further mentioned. We refer the interested reader to \cite{BBP1} for the details on this truncation argument.

The following elementary observation will be used repeatedly in the sequel. If $B_1,\ldots, B_n$ are measurable subsets of $C([0,T];L^2)$ then
\begin{equation}
\label{bbb}
\eps\lambda_\eps\log \bb P_\eps\Big(\bigcup_{i=1}^n B_i\Big) \le \eps\lambda_\eps\log n + \bigvee_{i=1}^n \eps\lambda_\eps\log\bb P_\eps(B_i)\;.
\end{equation}

Hereafter, we shall denote by $C$ a generic positive constant, independent of $\eps$, whose numerical value may change from line to line and from one side to the other in an inequality.

\subsection{Energy estimate}
\label{sec:3.1}

In the context of the analysis of the action functional \cite{MR}, from the equi-bounded\-ness of the action it is deduced a uniform bound for the free-energy functional $\mc F_\eps$ given by \eqref{F} and the time integral of the diffuse Willmore functional $\mc W_\eps$ defined by \eqref{W}. In the stochastic setting, both the free energy and the time integral of the diffuse Willmore functional can be  arbitrarily large, however - as we here show - this happens with probability super-exponentially small. 
\begin{prop}
  \label{prop:1}
  Let $\bb P_\eps$ be the law of the solution to \eqref{1} with
  initial datum $\bar u^\eps_0$. 
  Then there exists a constant $\ell_0\in [1,+\infty)$ and $\eps_0>0$ such that for any $\eps \in (0,\eps_0]$ and
  $\ell\in [\ell_0,+\infty)$, 
  \begin{equation}
    \label{3}
     \eps\lambda_\eps\,
    \log\bb P_\eps\Big(\sup_{t\in [0,T]} \mc F_\eps(u_t) + \int_0^T\!
      \mc W_\eps(u_t)\,\rmd t > \ell\Big) \le - \frac{\ell}{20}\;. 
  \end{equation}
\end{prop}
We start by a general martingale inequality that generalizes the Bernstein inequality, see, e.g., \cite[Ex.~VI.3.16]{RY}, which is obtained by choosing $\beta=0$ in Lemma~\ref{mm} below. The next statement is a particular case of \cite[Lemma 2]{Mariani} to which we refer for the proof.

\begin{lem}
  \label{mm}
Let $M$ be a real, continuous, square integrable martingale starting from $0$ with quadratic variation $[M]$. Given $\beta\ge 0$ and $C\in (0,+\infty)$, for any bounded stopping time $\tau$,
  \begin{equation*}
  \bb P \Big( \sup_{t\le \tau} M_t > \ell \:,\: [M]_\tau \le \beta
  \, \sup_{t\le \tau} M_t + C \Big) \le \exp\Big\{ - \frac
  {\ell^2}{2(\beta \ell +C)}\Big\}\;, \qquad \ell >0\;.  
  \end{equation*}
\end{lem}

\noindent\textit{Proof of Proposition \ref{prop:1}}\hspace{2truept} 
By It\^o's formula, with $\bb P_\eps$-probability one, for each $t\in [0,T]$,
  \begin{equation}
  \label{ito2}
  \mc F_\eps(u_t) + \int_0^t\!\mc W_\eps(u_s)\,\rmd s 
  = \mc F_\eps(\bar u_0^\eps) + R_t + N_t\;, 
  \end{equation} 
  where $N$ is a continuous $\bb P_\eps$-martingale and the It\^o's term is
  \[
    R_t = \eps\lambda_\eps \int_0^t\!\int\!\int\!
    \big[\nabla j_\eps(x-y)\big]^2
    \,\rmd x\,\rmd y\, \rmd s + \lambda_\eps
    \int_0^t\!\int\! \frac 1\eps (j_\eps*j_\eps)(0) W''(u)
    \,\rmd x\,\rmd s\;.
  \]
  Therefore, for $\eps_0$ small enough (depending on $T$) and $0<\eps \le \eps_0$,
  \begin{equation}
  \label{rtnew}
  |R_t|  \le \eps\lambda_\eps \|\nabla
  j_\eps\|_{L^2}^2\,t  + \eps^{-1}\lambda_\eps \|j_\eps\|_{L^2}^2 \int_0^t\!\int\! 
  |W''(u)|\,\rmd x\, \rmd s 
   \le C+ \frac12 \sup_{s \le t} \mc F_\eps(u_s) \;,
  \end{equation}
  where we used that $|W''|\leq C(1+\eps^{-1}W)$ by Assumption \ref{t:ws} and that, in view of \eqref{2.5b}, $\eps^{-1}\lambda_\eps \|j_\eps\|_{L^2}^2 \to 0$ as $\eps \to 0$.
  
  By taking the supremum over time in \eqref{ito2} and using the previous bound we get
   \begin{equation}
  \label{ito21}
  \frac12 \sup_{s\le t} \mc F_\eps(u_s) + \int_0^t\!\mc W_\eps(u_s)\,\rmd s 
  \le \mc F_\eps(\bar u_0^\eps) + \sup_{s \le t} N_s \leq \bar C +\sup_{s \le t} N_s \; ,
  \end{equation}  
where $\bar{C}:=\sup_{\eps \leq \eps_0} \mc F_\eps(\bar u_0^\eps)<+\infty$. The quadratic variation of $N$ is
  \begin{equation}
  \label{4}
  \begin{split} 
    [N]_t & = 2\lambda_\eps 
    \int_0^t\! \Big\|j_\eps * 
    \Big(\eps\Delta u_s - \frac 1\eps W'(u_s)\Big)\Big\|_{L^2}^2\,
    \rmd s \\ 
    & \le 2 \eps\lambda_\eps 
    \int_0^t\!\mc W_\eps (u_s)\, \rmd s 
    \le 2 \eps\lambda_\eps 
    \big[ \bar C  +  \sup_{s \le t}N_s \big]\;,
  \end{split}
  \end{equation}
  where we used \eqref{ito21} in the last inequality. By applying Lemma~\ref{mm} we deduce
  \[
   \eps\lambda_\eps\,
  \log\bb P_\eps\Big(\sup_{t\in [0,T]} N_t > \ell\Big) 
  \le - \frac{\ell}{4+4 \bar{C} \ell^{-1}} \le -\frac{\ell}{5}\;,
  \]
  provided $\ell \geq \ell_0:=4 \bar{C}+1$.
  Using again \eqref{ito21} the conclusion follows.
  \qed

\subsection{Continuity moduli}
\label{sec:3.2}

In this subsection we prove the estimates on the continuity moduli needed for the exponential tightness and to ensure that the rate function is finite only on the set $\bs\Gamma$ of admissible pairs, recall items c) and d) in Definition~\ref{def:2}.

\begin{prop}
  \label{t:omep}
  Let $\bb P_\eps$ be the law of the solution to \eqref{1} with
  initial datum $\bar u^\eps_0$.
  For each $\gamma \in (0,\frac{1}{4d})$ there exist constants $\eps_0>0$, $0<\delta_0 \le 1 \wedge T$, and $C_0\ge 1$ such that the following holds. For any $\eps \in (0,\eps_0]$, for any $\delta\in (0, \delta_0)$, and for any $\zeta >0$ such that $\zeta \delta^{-\gamma} \geq C_0\ell_0$ with $\ell_0$ as in Proposition \ref{prop:1},
 we have
  \begin{equation*}
       \eps\lambda_\eps\, \log\bb
      P_\eps\Big(\omega^\infty(u;\delta) >\zeta\Big) 
      \le - \,\frac { \zeta }{ C_0\delta^{\gamma}}\;,
  \end{equation*}
where $\omega^\infty(u;\delta) $ is the continuity modulus defined in \eqref{ominf}.
\end{prop}

The proof of the previous bound relies on the following lemma.

\begin{lem}
  \label{t:ome}
  Let $\bb P_\eps$ be the law of the solution to \eqref{1} with
  initial datum $\bar u^\eps_0$.  
  Let also $G\colon \bb R\to\bb R$ be defined by $G(u) = \int_0^u\!
  \sqrt{2W(v)} \, \rmd v$ and, given $\phi\in L^\infty$, set
  $z^\phi_t:= \int\! G(u_t) \phi \, \rmd x$, $t\in[0,T]$.
  For each $\gamma \in (0,\frac12)$ there exist constants $\eps_0>0$, $0<\delta_0 \le 1 \wedge T$, and $C_0\in (0,+\infty)$ such that the following holds. For any $\eps \in (0,\eps_0]$, any $\delta\in (0, \delta_0)$, any $\zeta >0$ such that $\zeta \delta^{-\gamma} \geq \ell_0$, and any $\phi\in L^\infty$, $\| \phi\|_\infty=1$,  
  \begin{equation*}
      \eps\lambda_\eps\, \log\bb
      P_\eps\Big(
      \sup_{|t-s| \le\delta} |z^\phi_t-z^\phi_s| >\zeta\Big) 
      \le - \,\frac { \zeta }{ C_0 \delta^\gamma}\;.
  \end{equation*}
\end{lem}

\begin{proof}
  By a simple inclusion of events, see e.g., the proof of Thm.\ 8.3 in \cite{Bi}, it
  is enough to show that 
  \begin{equation}
    \label{pz}
     \sup_{s\in [0,T-\delta]}
    \eps\lambda_\eps\, \log\frac T\delta \bb P_\eps\Big(\sup_{t\in
      [s,s+\delta]} |z^\phi_t-z^\phi_s| >\zeta  \Big) 
      \le - \,\frac { \zeta }{ C  \delta^\gamma}\;,
  \end{equation}
  for some constant $C\in(0,+\infty)$ independent on $\phi$, $\zeta$, and
  $\delta$.

  By It\^o formula, with $\bb P_\eps$-probability one, for each
  $s\in [0, T-\delta]$ and $t\in [s,s+\delta]$,
  \begin{equation}
    \label{zp}
      z^\phi_t - z^\phi_s = 
      D^{\phi,s}_t
      + R^{\phi,s}_t + N^{\phi,s}_t\;,
  \end{equation}
  where 
  \[
 D^{\phi,s}_t := \int_s^t\! \int\! \sqrt{2W(u)}
      \Big(\Delta u - \frac 1{\eps^2} W'(u)\Big) \phi \,
      \rmd x\, \rmd r\;, 
  \]
  the It\^o term is
  \[
  R^{\phi,s}_t = \lambda_\eps \int_s^t\!\int\! 
  \frac{ W'(u) \phi}{\sqrt{2W(u)}} 
  (j_\eps*j_\eps)(0) \,\rmd x\,\rmd r\;,
  \]
  and $N^{\phi,s}_t$, $t\in [s,T]$, is a $\bb P_\eps$-martingale with quadratic variation,
  \[
    [N^{\phi,s}]_t = 
    4\lambda_\eps \int_s^t\! 
    \Big\|j_\eps * \big(\phi \sqrt{W(u_r)} \big)\Big\|_{L^2}^2\,\rmd r 
    \le 4 \eps\lambda_\eps\delta  
    \sup_{r\in [0,T]}\mc F_\eps(u_r)\;.
  \]
  To control the martingale part, given $\ell>0$, we bound,
  \[
  \begin{split}
& \bb P_\eps\Big(\sup_{t\in [s,s+\delta]} |N^{\phi,s}_t| > \zeta \Big) \le \bb P_\eps\Big(\sup_{r\in [0,T]}\mc F_\eps(u_r) > \ell\Big) \\ & \qquad\qquad + \bb P_\eps\Big(\sup_{t\in [s,s+\delta]} |N^{\phi,s}_t| > \zeta\;,\;\sup_{r\in [0,T]}\mc F_\eps(u_r)\le \ell\Big)\;.
\end{split}
   \]
Choosing $\delta_0$ small enough so that $\ell = 2 \zeta/\sqrt \delta > \ell_0$, applying Proposition \ref{prop:1}, and Lemma~\ref{mm} with $\beta=0$ (both to the martingales $N^{\phi,s}$ and $-N^{\phi,s}$) we deduce, recalling \eqref{bbb}, that there exists a constant $C>0$ such that, for all $\delta$ small enough,
  \begin{equation}
    \label{pz1}
     \sup_{s\in
      [0,T-\delta]} \eps\lambda_\eps\, \log  \bb
    P_\eps\Big(\sup_{t\in [s,s+\delta]} |N^{\phi,s}_t| >\zeta  \Big)
    \le \eps \lambda_\eps \log 3- \frac{\zeta}{C\sqrt\delta} \;.
  \end{equation}

Concerning the It\^o term, for any $\eps\in (0,1)$,
\[
\begin{split}
|R^{\phi,s}_t| & \le \lambda_\eps \|j_\eps\|_{L^2}^2\, \delta \sup_{r\in [0,T]} \int\! \frac{|W'(u_r)|}{\sqrt{2W(u_r)}}\,\rmd x \\ & \le \lambda_\eps \|j_\eps\|_{L^2}^2\, \delta  C \sup_{r\in [0,T]} \int\! (1+W(u_r))\,\rmd x\; \\
& \leq C \, \delta \Big(1+ \sup_{r \in [0,T]} \mc F_\eps (u_r) \Big) \; , 
\end{split}
\]
where we used $|(2W)^{-1/2}W'|\le C(1+W)$ and the assumption \eqref{2.5b} on $j_\eps$. 

By choosing $\delta_0>0$ so small that $(C \sqrt \delta)^{-1} \geq 2+2 \ell^{-1}_0$ we have $\frac{\zeta}{C\delta} -1 \geq \frac{2 \zeta}{\sqrt \delta}$, hence, by Proposition \ref{prop:1} with $\ell=2\zeta /\sqrt{\delta}$ we obtain
\begin{equation}
\label{pz2}
 \eps \lambda_\eps\, \log\bb P_\eps\Big(\sup_{t\in [s,s+\delta]} |R^{\phi,s}_t| > \zeta \Big)  \leq  -\frac {\zeta}{10 \sqrt{\delta}}\;.
\end{equation}
Finally, by Cauchy-Swartz inequality, 
\[
\begin{split}
|D^{\phi,s}_t| & \le \sqrt\delta  \Big(\sup_{r\in [0,T]} \int\! \eps^{-1} 2 W(u_r)\,\rmd x\Big)^{\frac 12} \Big(\int_0^T \mc W_\eps(u_r) \,\rmd r\Big)^{\frac 12} \\ & \le \sqrt{\frac \delta 2}\, \Big(\sup_{r\in [0,T]} \mc F_\eps(u_r) + \int_0^T\! \mc W_\eps(u)\,\rmd r\Big)\;,
\end{split}
\] 
where in the last step we used Young's inequality. Hence,
  \[
 \bb P_\eps\Big(\sup_{t\in [s,s+\delta]} |D^{\phi,s}_t| > \zeta \Big) \le \bb P_\eps\Big(\sup_{r\in [0,T]}\mc F_\eps(u_r)+ \int_0^T\! \mc W_\eps(u)\,\rmd r > \zeta \sqrt{\tfrac{2}{\delta}} \,\Big) \; . 
   \]
Choosing $\delta_0$ small enough so that $\ell =  \zeta \sqrt{2 /\delta} > \ell_0$, applying Proposition~\ref{prop:1} with $\ell=\zeta \sqrt{2/\delta} \geq \ell_0$, we deduce
  \begin{equation}
    \label{pz3}
     \sup_{s\in
      [0,T-\delta]} \eps\lambda_\eps\, \log  \bb
    P_\eps\Big(\sup_{t\in [s,s+\delta]} |D^{\phi,s}_t| >\zeta  \Big)
    \le - \frac{\zeta}{20 \sqrt\delta} \;.
  \end{equation}
Since $\eps \lambda_\eps \to 0$, $ \log \frac{3}{\delta}=o (\delta^{-1/2+\gamma})$, and $\zeta \delta^{-\gamma} \geq \ell_0$, choosing $\eps_0>0$ small enough, the bound \eqref{pz} follows from \eqref{zp}, \eqref{pz1}, \eqref{pz2}, and \eqref{pz3}. 
\qed\end{proof}

To deduce Proposition \ref{t:omep} from the previous lemma, we need a rough ``measure of compactness" for the embedding $BV\hookrightarrow L^1$. 

\begin{lem}
\label{cover}
Let $\mc K$ be the subset of $L^1$ given by $\mc K := \big\{ v\in BV(\bb T^d) \colon \|v\|_{BV} \le 1 \big\}$. There exists a constant $C>0$ for which the following holds. For each $\sigma \in (0,1]$ there exists a finite set $\{v_1, \ldots, v_{N_\sigma} \}$ such that $\mc K \subset \bigcup_{i} \{ v \in L^1 \colon  \|v-v_i\|_{L^1}<\sigma \}$ and $N_\sigma \leq C (\sigma^{-d})^{\sigma^{-d}}$.
\end{lem}

\begin{proof}
Given $\sigma \in (0,1]$ and $m_0 \in \bb N$ to be fixed later, we let $n = m_0 \lfloor 1/\sigma \rfloor$ and write the fundamental domain $Q=[0,1)^d$ of the torus $\mathbb{T}^d$ as disjoint union of $n^d$ cubes $Q^n_{i}$ of linear size $1/n$ corresponding to multi-indices $i \in \{ 0, \ldots, n-1\}^d$.

Given $f \in \mc K$ let $f_i:= |Q^n_i|^{-1} \int_{Q^n_i} f\, \rmd x$ the average on each cube and $f^n:= \sum_i f_i \chi_{Q^n_i}$ the piecewise constant approximation of $f$.
By Holder, Sobolev, and Sobolev-Poincar\'e inequalities, with $1/1^*=1-1/d$, for any $i$ we have,
\[ 
|f_i| \leq |Q^n_i|^{-1/1^*}   \|f \|_{L^{1^*}(Q)} \leq  C n^{d-1}  \| f \|_{BV} \leq C n^{d-1} \;, \]
\[ \| f- f_i\|_{L^1(Q^n_i)} \le |Q^n_i|^{1-1/1^*}   \|f -f_i\|_{L^{1^*}(Q^n_i)} \leq  \frac Cn   \|Df \|_{TV(Q^n_i)} \; , 
\]
hence
\begin{equation}
\label{stap}
\| f- f^n\|_{L^1(Q)} \leq \frac Cn \;, \qquad \| f_n \|_{L^\infty(Q)} \leq C n^{d-1} \; .
\end{equation}
We let $n^{-1}\bb Z = \{m/n\;;\; m\in \bb Z\}$ and, for each $f^n$, we introduce its discrete approximation $\tilde f^n \colon Q \to n^{-1}\bb Z$ by setting $\tilde f^n = n^{-1} \sum_i \lfloor n f_i \rfloor \chi_{Q^n_i}$. Clearly,
\begin{equation}
\label{stap1}
\| f^n - \tilde f^n\|_{L^\infty(Q)} \leq \frac 1n \;, \qquad \| \tilde f_n \|_{L^\infty(Q)} \leq C n^{d-1} \; .
\end{equation}
By \eqref{stap} and \eqref{stap1} we have $\| f - \tilde f^n\|_{L^1(Q)} \le Cn^{-1} < \sigma$ for $m_0$ large enough, uniformly with respect to $f$. By construction, $\tilde f^n \in \mc Q_n$ with
\[
\mc Q_n := \Big\{g= \sum_i g_i \chi_{Q^n_i} \colon g_i\in n^{-1}\bb Z \cap [-Cn^{d-1},Cn^{d-1}]\Big\}\;.
\] 
As the cardinality $\mc Q_n$ is at most $C (n^d)^{n^d}$, the conclusion follows.  
\qed\end{proof}

\noindent\textit{Proof of Proposition~\ref{t:omep}}\hspace{2truept} 
Let $G$ be the function defined in Lemma~\ref{t:ome}. We claim that for each $\gamma \in (0,\frac1{2d})$ there exist constants $\eps_0>0$, $\delta_0 \in (0,T]$, and $C_1\ge 1$ such that the following holds. For any $\eps \in (0,\eps_0]$, for any $\delta\in (0, \delta_0)$, and for any $\zeta >0$ such that $\zeta \delta^{-\gamma} \geq C_1\ell_0$ with $\ell_0$ as in Proposition \ref{prop:1},
 we have
  \begin{equation}
  \label{Gm}
       \eps\lambda_\eps\, \log\bb
      P_\eps\Big(\omega^\infty(G(u);\delta) >\zeta\Big) 
      \le - \,\frac { \zeta }{ C_1\delta^{\gamma}}\;.
  \end{equation}

We first show that the claim implies the statement of the proposition. Since $W$ has quadratic minima and it has at least quadratic growth, then there exists $C>0$ such that $|G^{-1}(a)-G^{-1}(b)| \le C \sqrt{|a-b|}$ for any $a,b\in\bb R$. Hence, 
\[
\|u_t-u_s\|_{L^1} \le C \sqrt {\|G(u_t)-G(u_s)\|_{L^1}}\;,
\]
which implies the inclusion $\{\omega^\infty(u;\delta)>\zeta\} \subset \{\omega^\infty(G(u);\delta)>(\zeta/C)^2\}$. By applying \eqref{Gm}, with $\gamma \to 2\gamma$, $\zeta \to (\zeta/C)^2$, so that $\zeta^2\delta^{-2\gamma}\ge C^2 C_1\ell_0$ and $\gamma\in (0,\frac{1}{4d})$, the statement follows with $C_0 = \max\{1;C\sqrt{C_1/\ell_0}\}$.

We are left with the proof of \eqref{Gm}. To this end, we observe that $|\nabla G(u)|\le \frac\eps2 |\nabla u|^2 + \frac1\eps W(u)$ and, in view of Assumption \ref{t:ws}, $|G| \le C_2 (W+1)$ for some $C_2\ge 1$. Proposition~\ref{prop:1} thus implies that, for any $\ell\in [2C_2\ell_0,+\infty)$,
  \begin{equation}
    \label{3bb}
     \eps\lambda_\eps\,
    \log\bb P_\eps\Big(\sup_{t\in [0,T]} \big(\|G(u_t)\|_{L^1} + \|\nabla G(u_t)\|_{L^1} \big)   > \ell\Big) \le - \frac{\ell}{40C_2}\;.
  \end{equation}

Set $\mc K_\ell := \big\{ v\in BV(\bb T^d) \colon \|v\|_{BV} \le \ell\big\}$. For each $\rho>0$, by Lemma \ref{cover} with $\sigma=\rho\ell^{-1}$, there exists a finite set $\{v_1,\ldots,v_{N_{\rho\ell^{-1}}}\}\subset L^1$ such that $\mc K_\ell\subset \cup_i\{v\in L^1\colon \|v-v_i\|_{L^1} < \rho\}$ and $\log N_{\rho\ell^{-1}} \le C [1 + (\ell\rho^{-1})^d \log (\ell\rho^{-1})]$. Furthermore, for each $i,j=1,\ldots, N_{\rho\ell^{-1}}$ there is $\phi_{i,j}\in L^\infty$ of unit norm, given by $\phi_{i,j}=\mathrm{sgn} (v_i-v_j)$, such that $\|v_i-v_j\|_{L^1} = \int\! (v_i-v_j)\phi_{i,j}\,\rmd x$. 

Given $\gamma \in (0,\frac1{2d})$, choosing $\ell=\zeta\delta^{-\gamma} \geq 2C_2\ell_0$, \eqref{3bb} yields
 
 \begin{equation}
  \label{Gm1}
    \eps\lambda_\eps\, \log\bb
      P_\eps\Big( \big\{G(u_t)\in\mc K_{\zeta\delta^{-\gamma}} \;\forall\,t\in [0,T]\big\}^\mathrm{c} \Big) 
      \le - \frac { \zeta }{ C\delta^{\gamma}}\;. 
  \end{equation} 
  
Choosing $\rho=\zeta/5$, we have the inclusion of events 
\[
\begin{split}
\big\{\omega^\infty(G(u);\delta) >\zeta \big\} \cap \; \big\{G(u_t)\in\mc K_{\zeta\delta^{-\gamma}} \;\forall\,t\in [0,T]\big\} \\ \subset  \bigcup_{i,j} \Big\{
       \sup_{ |t-s|\le \delta} \int\! (G(u_t)-G(u_s))\phi_{i,j}\,\rmd x >\zeta/5\Big\} \; ,
\end{split}
\]
therefore by applying Lemma~\ref{t:ome} with exponent $\gamma'\in (0,\frac 12)$ and $\zeta\delta^{-\gamma'}\ge \ell_0$, and the above bound on $\log N_{\rho\ell^{-1}}$, we obtain, by choosing $\eps_0$ such that $\eps\lambda_\eps\le 1$ for $\eps\in [0,\eps_0]$, 
\begin{equation}
\label{Gm2}
\begin{split}
& \eps\lambda_\eps\,\log\bb P_\eps\Big(\omega^\infty(G(u);\delta) >\zeta \;, \; \big\{G(u_t)\in\mc K_{\zeta\delta^{-\gamma}} \;\forall\,t\in [0,T]\big\} \Big) \\ & \qquad \le \eps \lambda_\eps \log N_{\ell^{-1}\zeta/5}^2 - \frac { \zeta }{ C\delta^{\gamma'}} \le C (1 + \delta^{-d\gamma} \log \delta^{-\gamma}) - \frac{ \zeta }{ C\delta^{\gamma'}} \;. 
\end{split}
\end{equation} 
Since $\zeta\delta^{-\gamma} \ge 2C_2\ell_0$, by choosing $\gamma'\in (\gamma d,\frac12)$, we have $\zeta\delta^{-\gamma'}\ge \ell_0$ for any $\delta\in (0,\delta_0]$ and $\delta_0$ possibly smaller than the one in Lemma~\ref{t:ome}.

The bounds \eqref{Gm1} and \eqref{Gm2} then yield, for $\eps\in (0,\eps_0]$,
\[
\eps\lambda_\eps\, \log\bb P_\eps\Big(\omega^\infty(G(u);\delta) >\zeta\Big) \le \log 2 - \frac{\zeta}{C\delta^{\gamma}} \wedge \Big\{\frac {\zeta }{C\delta^{\gamma'}} -  C (1 + \delta^{-\gamma d} \log \delta^{-\gamma})\Big\}\;,
\]
which yields the claim for a possibly smaller choice of $\delta_0$.
\qed

\begin{prop}
\label{t:ene}
Let $\bb P_\eps$ be the law of the solution to \eqref{1} with initial datum $\bar u^\eps_0$. For $\phi\in C^1(\bb T^d)$ let $z^\phi_t := |V_{\eps,t}^u|(\phi)$. For each $\alpha\in (0,\frac 12)$ there exist constants $\eps_2 > 0$, $0<\delta_1 \le 1 \wedge T$, and $C_2\ge 1$ such that the following holds. For any $\phi\in C^1(\bb T^d)$ with $\|\phi\|_{C^1} \le 1$, any $\eps \in (0,\eps_2]$, any $\delta\in (0, \delta_1)$, and any $\zeta>0$ such that $\zeta \delta^{-\alpha} \geq C_2\ell_0$, with $\ell_0$ as in Proposition \ref{prop:1},
\begin{equation}
\label{z2}
\eps\lambda_\eps\, \log\bb P_\eps\big(\omega^1(z^\phi;\delta)>\zeta\big) \le -\frac{\zeta}{C_2\|\phi\|_{C^1}\delta^\alpha}\;,
\end{equation}
where $\omega^1(z;\delta)$ is the $L^1$-continuity modulus defined in \eqref{om1}.
\end{prop}

We start by a general compactness property for families of martingales. To put the following result in perspective, consider a family of continuous real martingales $M^\eps$, whose quadratic variation admits the bound $\rmd [M^\eps]_t \le C \eps \rmd t$. Then, a straightforward application of Berstein's inequality yields a super-exponential estimate for its continuity modulus in $C([0,T])$ (this is indeed the argument used in Lemma~\ref{t:ome}). Next, we consider instead the case in which the quadratic variation admits the bound $[M^\eps]_T \le C \eps T$ and deduce a super-exponential estimate for the continuity modulus in $L^1([0,T])$.

\begin{lem}
\label{lem:boh}
Given $T>0$ let $M^\eps= \{M^\eps_t\}_{t\in [0,T]}$, $\eps\in (0,\eps']$, be a family of real, continuous, square integrable $\bb P_\eps$-martingales starting from $0$ with quadratic variation $[M^\eps]$. If there exist $C'>0$ and $\ell'>0$ such that, for any $\eps\in (0,\eps']$, and $\ell\ge \ell'$,
\begin{equation}
\label{pp7}
\eps\,\log\bb P_\eps \big(\eps^{-1} [M^\eps]_T > \ell \big) \le -\frac{\ell}{C'}\;, 
\end{equation}
then, for each $\alpha\in (0,\frac 12)$ there is $C''\ge 1$ such that the following holds. For any $\zeta>0$ and $\delta\in (0,1\wedge T]$ with $\zeta\delta^{-\alpha} \ge C'' \ell'$, and any $\eps\in (0,\eps']$,
\begin{equation}
\label{es-1}
\eps \log\bb P_\eps\big(\omega^1(M^\eps;\delta)>\zeta\big) \le  - \frac{\zeta}{C''\delta^\alpha}\;.
\end{equation}
\end{lem}

\begin{proof}
By the representation of continuous martingales as time-changed Brownian motions, see, e.g., \cite[Chap.~V, Thm.~1.6]{RY}, 
\begin{equation}
\label{mb}
M^\eps \stackrel{\rm Law}= \eps^{\frac 12} B_{\tau^\eps}\;,
\end{equation}
where $B$ is a standard Brownian motion and $\tau^\eps_t = \eps^{-1} [M^\eps]_t$. By the Borell's inequality \cite[Thm.~2.1]{Adler}, for any $\alpha\in \big[0,\frac 12\big)$ and $S>0$, 
\begin{equation}
\label{es0}
\bb P\Big(\sup_{s,s'\in [0,S]} \frac{|B_s-B_{s'}|}{|s-s'|^\alpha} > \lambda\Big) \le 4 \exp \Big(-\frac{(\lambda-e_S)^2}{2\sigma_S} \Big)\;,  \qquad \lambda \ge e_S\;,
\end{equation}
where, using also the parabolic scale invariance of Brownian motion, 
\begin{equation}
\label{es}
\begin{split}
e_S & := \bb E\Big(\sup_{s,s'\in [0,S]} \frac{|B_s-B_{s'}|}{|s-s'|^\alpha}\Big)=S^{\frac 12 -\alpha} e_1 \;, \\ \sigma_S & := \sup_{s,s'\in [0,S]} \bb E \Big( \frac{|B_s-B_{s'}|^2}{|s-s'|^{2\alpha}}\Big) = S^{1-2\alpha}\sigma_1\;.
\end{split}
\end{equation}

Fix $\alpha\in (0,\frac 12)$, $\ell\ge\ell'$, and let
\[
\mc B_\ell := \Big\{\sup_{t,t'\in [0,T]} \frac{|M^\eps_t-M^\eps_{t'}|}{|\tau^\eps_t-\tau^\eps_{t'}|^\alpha} \le \lambda_\ell\Big\}\;,
\]
with $\lambda_\ell>0$ to be fixed below. We have,
\[
\bb P_\eps \big(\mc B_\ell^\mathrm{c}\big)  \le 2 \Big[ \bb P_\eps \big(\mc B_\ell^\mathrm{c}\cap\{\tau^\eps_T\le \ell\}\big) \vee \bb P_\eps \big(\tau^\eps_T > \ell\big) \Big]\;.
\]
By \eqref{mb} the first probability in the right-hand side can be bounded by using \eqref{es0} and \eqref{es} with $(\lambda,S)$ replaced by $(\eps^{-\frac 12}\lambda_\ell,\ell)$, while a bound for the second one is given by \eqref{pp7}. Therefore, choosing $\lambda_\ell = \tilde C \ell^{1-\alpha}$, $\ell\ge \tilde C  \ell'$ for a suitable $\tilde C =\tilde C (\eps',C')$ large enough we obtain, using \eqref{bbb},
\begin{equation}
\label{es1}
\log\bb P_\eps \big(\mc B_\ell^\mathrm{c}\big) \le \eps \log 4 + \bigg(\eps\log 4  - \eps \frac{(\eps^{-\frac 12}\lambda_\ell - \ell^{\frac 12 - \alpha} e_1)^2}{2\ell^{1-2\alpha}\sigma_1}\bigg) \vee \bigg(- \frac{\ell}{C'}\bigg) \le  - \frac{\ell}{2C'} \;.
\end{equation}
By using that $M^\eps_0=0$, the monotonicity of $t\mapsto \tau^\eps_t$, the concavity of $x\mapsto x^\alpha$, and Jensen inequality, a straightforward computation yields, 
\[
\omega^1(M^\eps;\delta) \le 2\lambda_\ell \delta (\tau^\eps_T)^\alpha + \lambda_\ell T^{1-\alpha}\delta^\alpha (\tau^\eps_T)^\alpha \quad \mbox{on the event $\mc B_\ell$,}
\]
which implies, as $\lambda_\ell = \tilde C \ell^{1-\alpha}$, 
\[
\omega^1(M^\eps;\delta) \le \tilde C (2\delta + T^{1-\alpha}\delta^\alpha)\ell \le C\delta^\alpha\ell \quad \mbox{on the event $\mc B_\ell\cap \{\tau^\eps_T\le \ell\}$.}
\]
Hence, given $\zeta>0$ and choosing $\ell = C \zeta \delta^{-\alpha}$, the set $\{\omega^1(M^\eps;\delta) >\zeta\}$ is contained in $\mc B_\ell^\mathrm{c}\cup \{\tau^\eps_T>\ell\}$. By choosing $C''\ge 1$ large enough, the estimate \eqref{es-1} follows by \eqref{pp7} and \eqref{es1}.
\qed\end{proof}

\noindent\textit{Proof of Proposition \ref{t:ene}}\hspace{2truept} 
By Ito's formula and \eqref{1}, 
\begin{equation}
\label{pp2}
z^\phi_t = |V_\eps^{\bar u_\eps^0}|(\phi) + D^{\phi,1}_t + D^{\phi,2}_t + N^\phi_t + R^\phi_t\;, 
\end{equation}
where 
\[
\begin{split}
D^{\phi,1}_t & := -\int_0^t\! \int\! \nabla\phi\cdot\nabla u \Big(\eps\Delta u - \frac 1\eps W'(u)\Big) \, \rmd x\, \rmd s\;, \\ D^{\phi,2}_t & := \frac 1 \eps  \int_0^t\! \int\! \phi \Big(\eps\Delta u - \frac 1\eps W'(u)\Big)^2 \, \rmd x\, \rmd s\;,  \end{split}
\]
and, after a few integration by parts, the It\^o term $R^\phi$ reads, 
\begin{equation}
\label{pp2.5}
\begin{split}
& R^\phi_t = \eps\lambda_\eps \int_0^t\!\int\!\int\! \phi(x)\big[\nabla j_\eps(x-y)\big]^2 \,\rmd x\,\rmd y\, \rmd s \\ & \qquad\qquad + \lambda_\eps \int_0^t\!\int\! \phi \frac 1\eps (j_\eps*j_\eps)(0) W''(u) \,\rmd x\,\rmd s\;.
\end{split}  
\end{equation}
Finally, $N^\phi$ is a $\bb P_\eps$-martingale with quadratic variation,
\[
\begin{split}
[N^\phi]_t & = 2\lambda_\eps \int_0^t\! \int\! \Big\{j_\eps * \Big[\eps\nabla\phi\cdot\nabla u + \phi  \Big(\eps\Delta u-\frac 1\eps W'(u)\Big) \Big]\Big\}^2\, \rmd x\,\rmd s \\ & \le 4\eps\lambda_\eps \int_0^t\! \int\! \Big[\eps |\nabla\phi|^2 |\nabla u|^2+ \phi^2 \, \frac 1\eps \Big(\eps\Delta u - \frac 1\eps W'(u)\Big)^2\Big]\, \rmd x\,\rmd s \\ & \le 4\eps\lambda_\eps  \Big[ 2T \|\nabla\phi\|_\infty^2 \sup_{s\in [0,T]} \mc F(u_s) + \|\phi\|_\infty^2 \int_0^T\! \mc W_\eps(u_t)\, \rmd t \Big]\;.
\end{split}
\]

By Proposition~\ref{prop:1}, this bound implies that there exists $C>0$ such that, for any $0<\eps< \eps_0$ and $\ell \geq 8(1+T) \ell_0$,
\begin{equation}
\label{N}
 \eps\lambda_\eps \log \bb P_\eps\big( (\eps\lambda_\eps)^{-1}[N^\phi]_T>\ell\big) \le - \frac{\ell}{C(\|\phi\|_\infty+\|\nabla\phi\|_\infty)^2}\;.
\end{equation}
By applying Lemma \ref{lem:boh} to the family of martingales $\{ \| \phi \|^{-1}_{C^1} N^\phi \}_\eps$ there is $C''\ge 1$ such that the following holds. For any $\zeta>0$ and $\delta\in (0,1\wedge T]$ with $\zeta\delta^{-\alpha} \ge 4 C''  8(1+T) \ell_0$, and any $\eps\in (0,\eps_0]$,
\begin{equation}
\label{es-2}
\eps \lambda_\eps \log\bb P_\eps\big(\omega^1(N^\phi;\delta)>\zeta/4\big) \le  - \frac{\zeta}{4C'' (\|\phi\|_\infty+\|\nabla\phi\|_\infty)\delta^\alpha}\;.
\end{equation}

We next estimate the second and third term on the right-hand side of \eqref{pp2}. On one hand, 
\[
\int_0^\delta \!\big ( \big|  D^{\phi,2}_t\big|+ \big|  D^{\phi,2}_{T-t}\big| \big)\, \rmd t + \int_\delta^T \!\big| D^{\phi,2}_t - D^{\phi,2}_{t-\delta} \big|\, \rmd t \le 3\delta \|\phi\|_\infty \int_0^T\! \mc W_\eps(u_t)\, \rmd t\;.
\]
On the other hand, by Young's inequality, 
\[
\begin{split}
\int_0^\delta \! \big(\big |  D^{\phi,1}_t\big|+ \big|  D^{\phi,1}_{T-t}\big| \big)\, \rmd t & \le  2\delta \|\nabla\phi\|_\infty \int_0^T\! \int\! \Big[\eps |\nabla u|^2 + \frac 1\eps \Big(\eps\Delta u - \frac 1\eps W'(u)\Big)^2\Big]\, \rmd x\, \rmd t \\ & \le 2\delta \|\nabla\phi\|_\infty \Big[ 2T \sup_{s\in [0,T]} \mc F_\eps(u_s) + \int_0^T\! \mc W_\eps(u_t)\, \rmd t \Big]\;,
\end{split}
\]
and, for the same reason,
\[
\int_\delta^T \!\big| D^{\phi,1}_t - D^{\phi,1}_{t-\delta} \big|\, \rmd t \le \delta \|\nabla\phi\|_\infty \Big[ 2T \sup_{s\in [0,T]} \mc F(u_s) + \int_0^T\! \mc W_\eps(u_t)\, \rmd t \Big]\;.
\]
Since $\omega^1(D^{\phi,1}+D^{\phi,2};\delta) \le \omega^1(D^{\phi,1}; \delta)+\omega^1(D^{\phi,2};\delta)$, by Proposition \ref{prop:1} and the previous estimates we conclude there exists $C>0$ such that for any $\zeta>0$ and $\delta\in (0,1\wedge T]$ with $\zeta\delta^{-1} \ge 2(1+2T) \ell_0$, and any $\eps\in (0,\eps_0]$,
\begin{equation}
\label{od}
\eps \lambda_\eps \log \bb P_\eps \big( \omega^1(D^{\phi,1}+D^{\phi,2};\delta)> \zeta/2 \big) \le  -\frac{\zeta}{C \delta (\|\phi\|_\infty+\|\nabla\phi\|_\infty)} \;.
\end{equation}
It remains to consider the Ito term. In view of \eqref{2.5b}, \eqref{pp2.5}, and the inequality $|W''|\leq C(1+W)$, there exists $\eps_2\leq \eps_0$ such that for any $\eps\in (0, \eps_2]$ we have,
\[ 
\omega^1(R^\phi;\delta) \leq C \delta (1+\| \phi\|_\infty  T \sup_{s\in [0,T]} 
\mc F_\eps(u_s) ) \; .
\]
Again by Proposition \ref{prop:1}, there exists $C\geq 1$ such that for any $\zeta>0$ and $\delta\in (0,1\wedge T]$ with $\zeta\delta^{-1} \ge 8C  (1+T) \ell_0$, and any $\eps\in (0,\eps_2]$,
\begin{equation}
\label{od1}
\eps \lambda_\eps \log \bb P_\eps \big( \omega^1(R^\phi;\delta)> \zeta/4 \big) \le  -\frac{\zeta}{C \delta \|\phi\|_\infty} \;.
\end{equation}
Combining \eqref{es-2}, \eqref{od}, and \eqref{od1}, a simple inclusion of events together with \eqref{bbb} implies that there exists $C\geq 1$ such that for any $\zeta>0$ and $\delta\in (0,1\wedge T]$ with $\zeta\delta^{-\alpha} \ge C \ell_0$, and any $\eps\in (0,\eps_2]$,
\[
\eps\lambda_\eps\, \log\bb P_\eps\big(\omega^1(z^\phi;\delta)>\zeta\big) \le \eps\lambda_\eps\log 3 - \frac{\zeta}{C\|\phi\|_{C^1}\delta^\alpha}\;.
\] 
Finally, since $\eps\lambda_\eps \to 0$ as $\eps\to 0$, by choosing a possibly smaller $\eps_2$ the claim \eqref{z2} follows for $C_2\ge 1$ large enough. 
\qed

\subsection{Bounds on the stochastic currents}
\label{sec:3.3b}

Recalling the definition of the stochastic currents in Subsection~\ref{sec:3.3}, we first prove that - with probability super-exponentially close to one - the stochastic current $J_\eps^u$ takes values in bounded subsets of $\bs H^{-\bs s}$. 

\begin{lem}
\label{lem:J3}
Given $\bs s\in (\frac 12,1)\times (\frac d2 ,+\infty)\times (\frac{d-1}2,+\infty)$,
\[
\lim_{\ell\to +\infty} \varlimsup_{\eps\to 0} \eps\lambda_\eps \log \bb P_\eps\big(\|J_\eps^u\|_{\bs H^{-\bs s}}^2 >\ell \big) = -\infty\;.
\]
\end{lem}

\begin{proof}
By Remark \ref{rem:boh},
\begin{equation}
\label{num3}
\|J_\eps^u\|_{\bs H^{-\bs s}}^2 \le C \big(\mc C_{1,\eps}^u + \mc C_{2,\eps}^u\big)\;,
\end{equation}
where
\[
\sqrt{\mc C_{1,\eps}^u} \le \int_0^T\! \int\! \frac\eps 2|\nabla u_t|^2\,\rmd x\,\rmd t + \int_0^T\, \mc W_\eps(u_t)\,\rmd t \le T \sup_{t\in[0,T]} \mc F_\eps(u_t) + \int_0^T\, \mc W_\eps(u_t)\,\rmd t
\]
and
\[
\mc C_{2,\eps}^u = \sum_{m,n,k} \int\! (1+n^2)^{-s_1} (1+|k|^2)^{-s_2}(1+|q|^2)^{-s_3} |Z_{n,T}^{\eps,m}(k,q)|^2\, \rmd q\;,
\]
with (see \eqref{eq:enk} for the definition of the functions $e_{n,k,q}^m$)
\[
Z_{n,t}^{\eps,m}(k,q) = \sqrt{2\eps\lambda_\eps} \int_0^t\! \Big\langle \sqrt\eps\nabla u_s \cdot e_{n,k,q}^m\big(s,\cdot,\tfrac{\nabla u_s}{|\nabla u_s|}\big)\;,\rmd \alpha^\eps_s \Big \rangle_{L^2} \;.
\]
By Proposition \ref{prop:1},
\begin{equation}
\label{ppr}
\lim_{\ell\to +\infty}\varlimsup_{\eps\to 0} \eps\lambda_\eps \log \bb P_\eps \Big(\mc C_{1,\eps}^u > \ell \Big) = -\infty\;.
\end{equation}
Set now
\[
\gamma := \sum_{m,n,k} \int\! (1+n^2)^{-s_1} (1+|k|^2)^{-s_2}(1+|q|^2)^{-s_3}\, \rmd q
\]
and introduce the probability measure $\Gamma$ on $\{1,\ldots,d\}\times\bb Z_+\times\bb Z^d\times \bb R^d$ defined by
\[
\Gamma (\rmd a) := \gamma^{-1} (1+n^2)^{-s_1} (1+|k|^2)^{-s_2}(1+|q|^2)^{-s_3}\, \rmd\wp(m,k,n) \,\rmd q\;, \;\; a = (m,n,k,q)\;,
\]
where $\wp$ is the counting measure on $\{1,\ldots,d\}\times\bb Z_+\times\bb Z^d$. Let also $e_a:=e_{n,k,q}^m$. Then,
\begin{equation}
\label{pps}
\mc C_{2,\eps}^u = \gamma \int\! Z^\eps_T(a)^2\,\Gamma(\rmd a)\;,  
\end{equation}
where $Z^\eps_t(a)=Z_{n,t}^{\eps,m}(k,q)$, $t\in [0,T]$, is a $\bb P_\eps$ continuous martingale. The \textit{bracket}, see, e.g., \cite[Chap.~IV, Def.~1.10]{RY}, between $Z^\eps(a)$ and $Z^\eps(b)$ is 
\begin{equation}
\label{ppa}
\begin{split}
& [Z^\eps(a),Z^\eps(b)]_t \\ & \;\; = 2\eps\lambda_\eps \int_0^t\! \eps \Big\langle j_\eps * \Big(\nabla u_s \cdot e_a\big(s,\cdot,\tfrac{\nabla u_s}{|\nabla u_s|}\big)\Big), j_\eps * \Big(\nabla u_s \cdot e_b\big(s,\cdot,\tfrac{\nabla u_s}{|\nabla u_s|}\big)\Big)\Big\rangle_{L^2}\,\rmd s\\ & \;\;  \le 2\eps\lambda_\eps \|e_a\|_\infty\|e_b\|_\infty\int_0^t\! \int\! \eps |\nabla u_s|^2\, \rmd x\,\rmd s \le C\eps\lambda_\eps \sup_{s\in[0,t]} \mc F_\eps(u_s) \;. 
\end{split}
\end{equation}
We next introduce the family of $\bb P_\eps$-martingales $Y^\eps_t(a) := Z^\eps_t(a)^2 - [Z^\eps(a)]_t$, $t\in [0,T]$, $a\in \{1,\ldots,d\}\times\bb Z_+\times\bb Z^d\times \bb R^d$. By a straightforward computation, the bracket between $Y^\eps(a)$ and $Y^\eps(b)$ is 
\[
\begin{split}
& [Y^\eps(a),Y^\eps(b)]_t = 4 \int_0^t\! Z^\eps_s(a)Z^\eps_s(b)\, \rmd  [Z^\eps(a),Z^\eps(b)]_s \\ & \qquad \le C \eps\lambda_\eps \sup_{s'\in[0,t]} \mc F_\eps(u_{s'}) \int_0^t\! \big( Z^\eps_s(a)^2+Z^\eps_s(b)^2\big) \,\rmd s \\ & \qquad\le C \eps\lambda_\eps \sup_{s'\in[0,t]} \mc F_\eps(u_{s'}) \int_0^t\! \big( Y^\eps_s(a)+Y^\eps_s(b)\big) \,\rmd s + C\Big(\eps\lambda_\eps \sup_{s'\in[0,t]} \mc F_\eps(u_{s'})\Big)^2\;.
\end{split}
\]
Now, the process $X^\eps_t :=\int\! Y^\eps_t(a)\,\Gamma(\rmd a)$ is still a $\bb P_\eps$ martingale with quadratic variation,
\begin{equation}
\label{ppx}
\begin{split}
[X^\eps]_t & = \int\! [Y^\eps(a),Y^\eps(b)]_t\,\Gamma(\rmd a)\,\Gamma(\rmd b) \\ & \le C \eps\lambda_\eps \sup_{s'\in[0,t]} \mc F_\eps(u_{s'}) \int_0^t\! X^\eps_s \,\rmd s + C\Big(\eps\lambda_\eps \sup_{s'\in[0,t]} \mc F_\eps(u_{s'})\Big)^2\\ &  \le C \eps\lambda_\eps \sup_{s'\in[0,t]} \mc F_\eps(u_{s'}) \sup_{s\in[0,t]} X^\eps_s + C\Big(\eps\lambda_\eps \sup_{s'\in[0,t]} \mc F_\eps(u_{s'})\Big)^2\;.
\end{split}
\end{equation}
Given $\ell,\ell'>0$ we write,
\[
\bb P_\eps \Big(X^\eps_T > \ell \Big) \le \bb P_\eps \Big(X^\eps_T > \ell \;, \;  \sup_{s\in[0,T]} \mc F_\eps(u_{s}) \le \ell' \Big) + \bb P_\eps \Big( \sup_{s\in[0,T]} \mc F_\eps(u_{s}) > \ell' \Big)\;. 
\]
The bound \eqref{ppx} and Lemma \ref{mm} imply 
\[
\varlimsup_{\eps\to 0} \eps\lambda_\eps \log \bb P_\eps \Big(X^\eps_T > \ell \;, \;  \sup_{s\in[0,T]} \mc F_\eps(u_{s}) \le \ell'\Big) \le -\frac{\ell}{2C\ell'}\;. 
\]
By using Proposition \ref{prop:1} and taking first the limit $\ell\to+\infty$ and then $\ell'\to+\infty$ we conclude that
\begin{equation}
\label{ppy}
\lim_{\ell\to +\infty}\varlimsup_{\eps\to 0} \eps\lambda_\eps \log \bb P_\eps \Big(X^\eps_T > \ell \Big) = -\infty\;.
\end{equation}
We finally observe that by \eqref{pps} and \eqref{ppa},
\[
\mc C_{2,\eps}^u = \gamma X^\eps_T + \gamma \int\! [Z^\eps(a)]_T\,\Gamma(\rmd a) \le \gamma X^\eps_T + C\eps\lambda_\eps \sup_{s\in[0,T]} \mc F_\eps(u_s)\;,
\]
hence, applying \eqref{ppy} and again Proposition \ref{prop:1},
\[
\lim_{\ell\to +\infty}\varlimsup_{\eps\to 0} \eps\lambda_\eps \log \bb P_\eps \Big(\mc C_{2,\eps}^u > \ell \Big) = -\infty\;,
\]
which, together with \eqref{num3} and \eqref{ppr} concludes the proof.
\qed\end{proof}

The next two lemmata will allow us to represent currents in terms of velocities of $L^2$-flows.

\begin{lem}
\label{lem:J1}
\[
\lim_{\ell\to \infty}\sup_f \varlimsup_{\eps\to 0}\, \eps\lambda_\eps\, \log\bb P_\eps\Big( J_\eps^u(f) -\frac 12\int_0^T\! V_{\eps,t}^u\big(|f_t|^2\big)\, \rmd t > \ell\Big) = -\infty\;,
\]
where the supremum is carried out over $f\in C^\infty([0,T]\times\bb T^d\times\Lambda_{d-1};\bb R^d)$.
\end{lem}

\begin{proof}
From the very definition \eqref{J}, $J_\eps^u(f) = A^f_T + N^f_T$ where 
\[
A^f_T = - \eps\int_0^T\! \int\! \nabla u_t\cdot f_t(\cdot,(\mathbf{n}^u)^\perp)\, \left[\Delta u - \frac1{\eps^2} W'(u) \right] \,\rmd x\,\rmd t 
\]
and $N^f$ is a continuous $\bb P_\eps$-martingale with quadratic variation 
\begin{equation}
\label{pap:1}
[N^f]_t = 2\eps^2\lambda_\eps \int_0^t\! \big\|j_\eps * \big(\nabla u_s\cdot f_s(\cdot,(\mathbf{n}^u)^\perp)\big)\big\|_{L^2}^2\, \rmd s \le 4\eps\lambda_\eps\int_0^t\! V_{\eps,s}^u\big(|f_s|^2\big)\, \rmd s \;,
\end{equation}
where we used that $\eps|\nabla u|^2\,\rmd x\le 2\, \rmd |V^u_{\eps}|$. 

By Young inequality, 
\[
A^f_T \le \frac 14\int_0^T\! V_{\eps,t}^u\big(|f_t|^2\big)\, \rmd t + 2\int_0^T\! \mc W_\eps(u)\,\rmd t \;,
\]
so that, by Proposition~\ref{prop:1},
\begin{equation}
\label{nume}
\lim_{\ell\to \infty}\sup_f \varlimsup_{\eps\to 0}\, \eps\lambda_\eps\, \log\bb P_\eps\Big( A^f_T -\frac 14\int_0^T\!V_{\eps,t}^u\big(|f_t|^2\big)\, \rmd t > \ell\Big) = -\infty\;.
\end{equation}
In view of \eqref{pap:1},
\[
N^f_t -\frac 14 \int_0^T\!V_{\eps,t}^u\big(|f_t|^2\big)\, \rmd t \le N^f_t - \frac{1}{16\eps\lambda_\eps}[N^f]_T\;.
\]
For $\beta>0$, by exponential Chebyshev inequality we get, 
\[
\bb P_\eps\Big(N^f_t -\frac 14 \int_0^T\!V_{\eps,t}^u\big(|f_t|^2\big)\, \rmd t>\ell \Big) \le \rme^{-\beta\ell/(\eps\lambda_\eps)}\bb E_\eps\exp\Big\{\frac{\beta} {\eps\lambda_\eps}N^f_t - \frac{\beta}{16(\eps\lambda_\eps)^2} [N^f]_T \Big\}\;.
\]
By choosing $\beta = 1/8$ and using that $\bb E_\eps\exp\big\{aN^f_t-\frac 12 a^2 [N^f]_T\big\} \le 1$ for $a\in\bb R$, the above displayed bound yields  
\[
\lim_{\ell\to \infty}\sup_f \varlimsup_{\eps\to 0}\, \eps\lambda_\eps\, \log\bb P_\eps\Big( N^f_t -\frac 14\int_0^T\!V_{\eps,t}^u\big(|f_t|^2\big)\, \rmd t > \ell\Big) = -\infty\;,
\]
which, combined with \eqref{nume}, concludes the proof.
\qed\end{proof}

In the following lemma we adopt the short notation introduced just after Remark \ref{rem:2}. 

\begin{lem}
\label{lem:J2}
\[
\lim_{\ell\to \infty}\sup_\psi \varlimsup_{\eps\to 0}\, \eps\lambda_\eps\, \log\bb P_\eps\Big( J_\eps^u(\nabla \psi) +\int_0^T\! |V_{\eps,t}^u|(\partial_t\psi_t)\, \rmd t > \ell\Big)= -\infty\;, 
\]
where the supremum is carried out over $\psi\in C^\infty_K((0,T)\times\bb T^d)$ such that $\|\psi\|_\infty \le 1$.
\end{lem}

\begin{proof}
By It\^o formula and \eqref{1}, after a few integrations by parts, 
\begin{equation}
\label{vva2b}
\begin{split}
|V_{\eps,T}^u|(\psi_T) - |V_{\eps,0}^u|(\psi_0) & = \int_0^T\! |V_{\eps,t}^u|(\partial_t\psi_t)\, \rmd t + J_\eps^u(\nabla \psi) +  N^\psi_T+R^\psi_T  \\& \quad - \, \frac 1\eps \int_0^T\!\int\!\psi \Big(\eps\Delta u - \frac 1\eps W'(u)\Big)^2 \,\rmd x\,\rmd t \;,
\end{split}
\end{equation}
where $R^\psi_t$ is defined as in \eqref{pp2.5} with $\phi$ replaced by $\psi$ and $N^\psi$ is the continuous $\bb P_\eps$-martingale given by 
\begin{equation}
\label{6b}
N^\psi_t = \int_0^t\! \Big\langle \psi_s\Big[-\eps\Delta u_s + \frac 1\eps W'(u_s)\Big] \,, \rmd u_s - \Big(\Delta u_s - \frac 1{\eps^2} W'(u_s) \Big)\,\rmd s\Big\rangle_{L^2}\;,
\end{equation}
whose quadratic variation is
\[
\begin{split}
& [N^\psi]_t = 2\lambda_\eps \int_0^t\! \Big\|j_\eps * \Big[ \psi_s \Big(\eps\Delta u_s-\frac 1\eps W'(u_s)\Big)\Big]\Big\|_{L^2}^2\,\rmd s \\ & \;\; \le 2\lambda_\eps \int_0^t\!\Big\|\psi_s \Big(\eps\Delta u_s-\frac 1\eps W'(u_s)\Big)\Big\|_{L^2}^2\,\rmd s \le  4\eps\lambda_\eps   \int_0^T\! \mc W_\eps(u_s)\, \rmd t \quad \mbox{(as $\|\psi\|_\infty\le 1$)}\;.
\end{split}
\]
By Proposition~\ref{prop:1}, Lemma~\ref{mm} with $\beta=0$, and the bound above,
\[
\lim_{\ell\to \infty}\sup_\psi \varlimsup_{\eps\to 0}\, \eps\lambda_\eps\, \log\bb P_\eps\Big(-N^\psi_T> \ell \Big) = -\infty\;.
\]
Arguing as in the proof of \eqref{rtnew} and recalling \eqref{2.5b}, for $\eps$ small enough we have,
\begin{equation}
\label{rtp}
\begin{split}
|R^\psi_T| & \le \big(\eps\lambda_\eps \|\nabla j_\eps\|_{L^2}^2\,T  + \eps^{-1}\lambda_\eps \|j_\eps\|_{L^2}^2\big) \Big(C + \frac12 \sup_{s \le T} \mc F_\eps(u_s)\Big) \\ & \le C + \frac12 \sup_{s \le T} \mc F_\eps(u_s)\;,
\end{split}
\end{equation}
where the constant $C$ does not depend on the choice of $\psi$ with $\|\psi\|_\infty\le 1$. Similarly, for such $\psi$,
\[
\frac 1\eps \int_0^T\!\int\!|\psi|\Big(\eps\Delta u - \frac 1\eps W'(u)\Big)^2 \,\rmd x\,\rmd t \le \int_0^T\! \mc W_\eps(u)\, \rmd t\;.
\]
Finally, $V_{\eps,T}^u(\psi_T) = V_{\eps,0}^u(\psi_0) =0$ as $\psi$ has compact support. Therefore, by \eqref{vva2b}, the proof is achieved gathering the above bounds and using Proposition~\ref{prop:1}.
\qed\end{proof}

\section{Large deviations upper bound}
\label{sec:4}

Recall the definitions of $\bs V$, $\bs U$, and $\bs H^{\bs s}$ in \eqref{V:=}, \eqref{U:=}, and \eqref{Hs:=}, and set 
\[
\bs Z:=\bs U\times\bs V\times\bs H^{-\bs s}\;,
\]
that we consider endowed with the product topology and the corresponding Borel $\sigma$-algebra. Note that, since $\bs U\times\bs H^{-\bs s}$ has a countable basis, then $\mc B(\bs Z) = \mc B(\bs U)\otimes \mc B(\bs V) \otimes \mc B(\bs H^{-\bs s})$, see \cite[Lemma 6.4.2]{Bogachev}. In this section we shall prove a large deviations upper bound for the family of probability measures on $\bs Z$ defined by $(\bb P_\eps \circ (Z_\eps^u)^{-1})$, where $Z_\eps^u := (u,V_\eps^u,J_\eps^u)$ is a Borel map according to Lemma \ref{mesura} and Theorem \ref{thm:sc}. Before stating the result we introduce the associated rate function. 

\begin{defin}
\label{def:d}
Let $\bs D$ be the subset of $\bs Z$ given by the collection of elements $Z= (u,V,J)$ such that:
\begin{itemize}
\item[a)] $(u,V) \in \bs\Gamma$ and $\tau^{-1}V$ is an $L^2$-flow.
\item[b)] The functional $J$ extends to a continuous linear functional on $L^2(V_t\,\rmd t;\bb R^d)$ satisfying $J(f) = \int_0^T\!\int\! f\cdot \nu \, \rmd V_t \,\rmd t$, $f \in L^2(V_t\,\rmd t;\bb R^d)$, where $\nu$ is a velocity of $\tau^{-1}V$.
\end{itemize}
\end{defin}
\begin{rem}
\label{rem:4}
The previous definition and the orthogonality condition \eqref{nut1} yield the inclusion $\bs L  \subset \mathrm{Ker} \, J $ for any $(u,V,J) \in \bs D$, where $\bs L \subset \bs H^{\bs s}$ is defined in Remark \ref{rem:2}. Note also that any functional $f\mapsto J(f)$ as in item b) of Definition \ref{def:d} with $\nu\in L^2(V_t\,\rmd t;\bb R^d)$ defines an element of $\bs H^{-\bs s}$.
\end{rem}

Given $\eta\in C^\infty([0,T]\times\bb T^d;\bb R^d)$, $\chi\in C^\infty_K([0,T)\times\bb T^d;\bb R^d)$, and $\psi\in C^\infty_K([0,T)\times\bb T^d;[0,1))$ such that $\sqrt{\psi(1-\psi)}\in C^\infty_K([0,T)\times\bb T^d)$, let $\mc I^{\eta,\chi,\psi}\colon \bs Z \to \bb R$ be the functional defined by
\begin{equation}
\label{mcI}
\begin{split}
& \mc I^{\eta,\chi,\psi}(u,V,J) := - \bar\mu_0(\psi_0) - J(\eta) - \int_0^T\! |V_t|(\partial_t\psi_t) \, \rmd t - J(\nabla\psi)\\ & + \int_0^T\! \delta V_t\big((2\psi_t-1)\eta_t+\chi_t\sqrt{\psi_t(1-\psi_t)}\big)\,\rmd t - \int_0^T\! |V_t|\Big(|\eta_t|^2+ \frac{|\chi_t|^2}4\Big) \, \rmd t \;,
\end{split}
\end{equation}
with $\bar \mu_0$ as in item c) of Assumption~\ref{t:au0}. Notice that, by definition of bounded weak* topology on $\bs V$, for each $\eta$, $\chi$ and $\psi$ the functions $Z \mapsto \mc I^{\eta,\chi,\psi}(Z)$ is continuous because it is sequentially continuous along weak* convergent sequences $V_n\to V$ and $\bs H^{-\bs s}$-convergent sequences $J_n \to J$.

Set 
\begin{equation}
\label{mci}
\mc I(Z) = \begin{cases} \sup_{\eta,\chi,\psi} \mc I^{\eta,\chi,\psi}(Z) & \text{ if } Z \in \bs D, \\ +\infty & \text{ otherwise,} \end{cases}
\end{equation}
where the supremum is carried out over $\eta\in C^\infty([0,T]\times\bb T^d;\bb R^d)$, $\chi\in C^\infty_K([0,T)\times\bb T^d;\bb R^d)$, and $\psi\in C^\infty_K([0,T)\times\bb T^d;[0,1))$ such that $\sqrt{\psi(1-\psi)}\in C^\infty_K([0,T)\times\bb T^d)$.
 
\begin{theorem}
\label{thm:in}
For each closed $\bs C\subset \bs Z$,
\[
\varlimsup_{\eps\to 0} \eps\lambda_\eps \log\bb P_\eps\big( Z_\eps^u\in \bs C\big) \le -\inf_{\bs C} \mc I \;.
\]
Moreover, $\mc I$ has compact sub-level sets. 
\end{theorem}
We first show that the above statement implies the large deviations upper bound.

\medskip
\noindent\textit{Proof of Theorem \ref{thm:1}} \hspace{2truept} 
By the contraction principle, see, e.g., \cite[Thm.~4.2.1]{DZ}, Theorem \ref{thm:in} implies the large deviations upper bound for the family $(\bb P_\eps \circ (u,V^u_\eps)^{-1})$ with good rate function 
\[
\bar I(u,V) = \inf\{\mc I(u,V,J)\;,\; J \in \bs H^{-\bs s}\}\;.
\]
It remains to show that $\bar I = I$ with $I$ as in \eqref{I=}. In view of \eqref{mci} and Definition \ref{def:d}, if $(u,V) \notin \Gamma$ or $\tau^{-1}V$ is not an $L^2$-flow then $\bar I(u,V) = I(u,V) = +\infty$ and the equality holds. Otherwise, $(u,V)\in \Gamma$, $\tau^{-1}V$ is an $L^2$-flow, and when computing $\bar I(u,V)$ we can assume $Z=(u,V,J) \in \bs D$, i.e., $J$ is given in terms of a velocity $\nu$ of $\tau^{-1}V$ as in item b) of Definition \ref{def:d}. For such $Z$ we have, 
\[
\begin{split}
& \mc I^{\eta,\chi,\psi}(Z) = - \int_0^T\! |V_t|(\eta_t \cdot \nu_t) \,\rmd t - \bar\mu_0(\psi_0) - \int_0^T\! |V_t|(\partial_t\psi_t+\nu_t\cdot\nabla\psi_t) \, \rmd t \\ & \!\! - \int_0^T\! |V_t| \big((2\psi_t-1)\eta_t\cdot H_t+\chi_t\cdot H_t\sqrt{\psi_t(1-\psi_t)} \big) \, \rmd t - \int_0^T\! |V_t| \Big(|\eta_t|^2+ \frac{|\chi_t|^2}4\Big) \rmd t \;.
\end{split}
\]
Since $C^\infty_K([0,T)\times\bb T^d;\bb R^d)$ is dense in $L^2([0,T]\times \bb T^d,|V_t|\,\rmd t;\bb R^d)$, the supremum $\sup_{\eta,\chi} \mc I^{\eta,\chi,\psi}(Z)$ is equal to the critical value $\mc I^{\eta_*,\chi_*,\psi}(Z)$, with $\eta_* = \frac 12(H-\nu) -\psi H$ and $\chi_*= - 2 \sqrt{\psi(1-\psi)} H$. A straightforward computation yields,
\[
\begin{split}
& \mc I^{\eta_*,\chi_*,\psi}(Z) = \frac 14 \int_0^T\! |V_t|\big(\big|\nu_t- H_t\big|^2\big) \,\rmd t \\ & \qquad \qquad\qquad - \bar\mu_0(\psi_0) - \int_0^T\! |V_t|\big(\partial_t\psi_t+\nu_t\cdot\nabla\psi_t - \nu_t\cdot H_t\psi_t\big) \, \rmd t\;.\end{split}
\]
Recalling \eqref{Iac}, \eqref{Ising}, and \eqref{I=} we thus conclude that if $(u,V)\in \Gamma$ and $\tau^{-1}V$ is an $L^2$-flow then, taking the infimum over $J$ (i.e., in view of Remark \ref{rem:4}, over all the possible velocities $\nu$ of $V$) we have,
\[
\bar I(u,V) = \inf_\nu \sup_\psi \mc I^{\eta_*,\chi_*,\psi}(Z) = \inf_\nu\{I_\mathrm{ac}(V,\nu) + I_\mathrm{sing}(V,\nu)\} = I(u,V)\;,
\]
which complete the proof.
\qed

\subsection{A priori bounds}
\label{subsec:4.1}

Fix a countable set $(f^{(k)})_{k\in \bb N}\subset C^\infty([0,T]\times\bb T^d\times \Lambda_{d-1};\bb R^d)$ dense in $\bs H^{\bs s}$. Given $N\in \bb N$ and $m>0$ let 
\[
\bs F_{m,N} := \Big\{ (V,J)\in \bs V\times\bs H^{-\bs s}\colon \max_{1\le k\le N}\Big[ J(f^{(k)}) -\frac 12\int_0^T\! V_t\big(|f^{(k)}_t|^2\big)\, \rmd t\Big] \le m\Big\}\;.
\]
Similarly, fix a countable set $(\psi^{(k)})_{k\in \bb N}$, with $\psi^{(k)} \in C^\infty_K((0,T)\times\bb T^d)$ such that $\|\psi^{(k)}\|_\infty\le 1$, dense in the unit ball of $C_0((0,T)\times\bb T^d)$, and let 
\[
\bs G_{m,N} := \Big\{ (V,J)\in \bs V\times\bs H^{-\bs s}\colon \max_{1\le k\le N} \Big[ J(\nabla \psi^{(k)}) +\int_0^T\! |V_t|(\partial_t \psi^{(k)}_t)\, \rmd t\Big] \le m\Big\}\;.
\]

Fix $\bs\ell\in\bb R^3_+$, recall Definition~\ref{def:2} of the set $\bs \Gamma_{\bs \ell}$ and denote by $\mc N(\bs\Gamma_{\bs\ell})$ the collection of the open neighborhoods of $\bs \Gamma_{\bs\ell}$. For $\mc A\in \mc N(\bs\Gamma_{\bs\ell})$, $m>0$, $N\in \bb N$, and recalling Remark~\ref{rem:2}, we set
\[
\begin{split}
\bs D_{\bs\ell,\mc A, m,N} := \Big\{ & (u,V,J) \in \bs Z \;:\;\; (u,V)\in \mc A\;, \\ &  (V,J)\in \bs F_{m,N} \cap \bs G_{m,N}\;,\;\; J(f)=0\;\; \forall\,f\in\bs L \Big\}\;.
\end{split}
\]

Recall the definition \eqref{dm} of the discrepancy measure $\xi^u_\eps$. The super-exponen\-tial  probability estimates in Propositions \ref{prop:1}, \ref{t:omep}, \ref{t:ene}, and Lemmata \ref{lem:J1}, \ref{lem:J2}, together with the deterministic bound in Lemma~\ref{prop:2} yields the following statement. 
 
\begin{prop}
\label{lem:zd}
For each $\zeta>0$, 
  \begin{equation*}
    \lim_{\eps\to 0}\, \eps\lambda_\eps\, 
    \log \bb P_\eps\Big(
    \int_0^T \!\big\| \xi^u_{\eps,t}\big\|_\mathrm{TV} \, 
    \rmd t > \zeta \Big) = -\infty\;.
  \end{equation*}
Moreover, 
\[
\lim_{\substack{\bs\ell\to\infty\\ m\to +\infty}} \sup_{\substack{N\in\bb N \\ \mc A\in \mc N(\bs\Gamma_{\bs\ell})}} \varlimsup_{\eps\to 0}\, \eps\lambda_\eps\, \log\bb P_\eps \big(Z_\eps^u \notin \bs D_{\bs\ell,\mc A, m,N} \big) = -\infty\;,
\]
where by $\bs\ell\to\infty$ we mean $\ell_i\to+\infty$, $i=1,2,3$. 
\end{prop}

\begin{proof}
To prove the first bound we observe that, in view of Lemma \ref{prop:2}, for each $\zeta>0$ and $\ell>0$ there exists $\eps_0$ such that for any $0<\eps<\eps_0$ we have the inclusion
 \[ \Big\{  u \colon \int_0^T \!\big\| \xi^u_{\eps,t}\big\|_\mathrm{TV} \, 
    \rmd t > \zeta \Big\} \subset  \Big\{  u \colon \sup_{t \in [0,T]} \mc F_\eps(u_t) +\int_0^T \mc W_\eps(u_s) ds >\ell \Big\} \, . \]
  Taking the limit $\eps \to 0$ and then $\ell \to \infty$, by Proposition \ref{prop:1} the first bound follows.
  
  To prove the second bound we first write,
  \[
  \begin{split} \Big\{ (u,V,J) \notin \bs D_{\bs\ell,\mc A, m,N} \Big\} & = \Big\{ (u,V) \notin \mc A \Big\} \cup  \Big\{ (V,J) \notin \bs F_{m,N} \Big\}\\ & \quad \cup  \Big\{ (V,J) \notin \bs G_{m,N} \Big\} 
  \cup \Big\{ \bs L \not\subset \mathrm{Ker} J\Big\}\; .
  \end{split}
  \]
Clearly,  by Remark \ref{rem:2} we have $\bb P_\eps (u \colon  \bs L \not\subset \mathrm{Ker} J^u_\eps )=0$. Moreover, by Lemmata \ref{lem:J1} and \ref{lem:J2} we easily deduce,
 \[
\lim_{ m\to +\infty} \sup_{N\in\bb N } \, \varlimsup_{\eps\to 0}\, \eps\lambda_\eps\, \log\bb P_\eps \big( (V_\eps^u,J_\eps^{u}) \notin \bs F_{ m,N} \big) = -\infty\;,
\]
and
\[
\lim_{ m\to +\infty} \sup_{N\in\bb N } \, \varlimsup_{\eps\to 0}\, \eps\lambda_\eps\, \log\bb P_\eps \big( (V_\eps^u,J_\eps^{u}) \notin \bs G_{ m,N} \big) = -\infty\;.
\]
In view of the previous bounds in order to conclude the proof it suffices to show that
\begin{equation}
\label{lunga}
\lim_{\bs\ell\to\infty} \sup_{ \mc A\in \mc N(\bs\Gamma_{\bs\ell})} \varlimsup_{\eps\to 0}\, \eps\lambda_\eps\, \log\bb P_\eps \big(
(u,V_\eps^u) \notin \mc A \big) = -\infty\;.
\end{equation}
Recalling Definition \ref{def:2}, for each fixed $\bs \ell \in \bb R^3_+$ and $\eps \in (0,1)$ we set
\[  \mc K_{\bs \ell,\eps} =  \bigcap_{\beta \in \{\mathrm{a,b,c,d} \}} \mc K_{\bs \ell,\eps}^{\beta} \, , \]
where, setting  $\bar{\delta}_k=\bar{\delta} 2^{-k}$ with $\bar{\delta}=\delta_0 \wedge \delta_1$, $\delta_0$ and $\delta_1$ as in Propositions \ref{t:omep} and \ref{t:ene}, 
\[ 
\begin{split}
\mc K_{\bs \ell,\eps}^{\mathrm a}&= \Big\{ u \in C([0,T];L^2) \colon  u_0=\bar u_0^\eps \Big\} \; ,\\
\mc K_{\bs \ell,\eps}^{\mathrm b}&= \Big\{ u \in C([0,T];H^1)\cap L^2([0,T];H^2)  \colon \\ & \qquad \qquad \sup_{t\in [0,T ]} \mc F_\eps (u_t) +\int_0^T \! \mc W_\eps(u_t) \, \rmd t  \leq \ell_1 \Big\} \; , \\
\mc K_{\bs \ell,\eps}^{\mathrm c}&= \bigcap_k \Big\{ u \in C([0,T];L^2)  \colon  \omega^\infty(u; \bar{\delta}_k) \leq  \frac{ \ell_2 {\bar \delta}^{1-\alpha_2}}{4 T^{1-\alpha_2}} \bar{\delta}_k^{\alpha_2}\Big\} \; , \\
\mc K_{\bs \ell,\eps}^{\mathrm d}&= \bigcap_{j=1}^{ \lfloor (\eps \lambda_\eps)^{-1} \rfloor } \bigcap_k \Big\{ u \in C([0,T];H^1)  \colon  \omega^1(|V^{u}_\eps|(\widehat\phi_j); \bar{\delta}_k) \leq \frac{\ell_3 \bar{\delta}^{1-\alpha_3}}{4 T^{1-\alpha_3}} \bar{\delta}_k^{\alpha_3} \Big\}\;,
\end{split}
\]
where $\widehat\phi_j :=\phi_j/\|\phi_j\|_{C^1}$. We claim that if there exist $ u^{\eps_n} \in \mc K_{\bs \ell, \eps_n}$ for a sequence $\eps_n \downarrow 0$, then there exists $(u, V) \in \bs \Gamma_{\bs \ell}$ such that, up to subsequences, $(u,V)=\lim_n (u^{\eps_n},V_{\eps_n}^{u^{\eps_n}})$ and $(u^{\eps_n})$ satisfies conditions a)-d) in Definition \ref{def:2}. It is straightforward to check that conditions a)-d) in Definition \ref{def:2} hold for the whole sequence $(u^{\eps_n})$. Indeed, a) and b) are trivial; moreover if $u^\eps \in \mc K^c_{\bs \ell,\eps}$ then $\omega^\infty(u^\eps; \delta) \leq \frac{ \ell_2{\bar \delta}^{1-\alpha_2}}{2 T^{1-\alpha_2}} \delta^{\alpha_2}$ for $0<\delta \leq \bar{\delta}$ since $\omega^\infty(\cdot;2\delta)\le 2\omega^\infty(\cdot;\delta)$ for $2\delta\le\bar\delta$;  whence c) follows easily, as $\omega^\infty(\cdot;\delta) \leq 2 \delta \bar \delta^{-1} \omega^\infty(\cdot;\bar \delta)$ for any $\bar \delta \leq \delta \leq T$. Similarly, since $\omega^1(\cdot;2\delta)\le 2\omega^1(\cdot;\delta)$ for $2\delta\le\bar\delta$, if $u^\eps \in \mc K^d_{\bs \ell,\eps}$ then $\omega^1(|V^{u}_\eps|(\widehat \phi_j); \delta) \leq \frac{\ell_3 \bar{\delta}^{1-\alpha_3}}{2 T^{1-\alpha_3}} \delta^{\alpha_3}$ for $0<\delta \leq \bar{\delta}$; whence  d) holds as $\omega^1(\cdot;\delta) \leq 2 \delta \bar \delta^{-1}  \omega^1(\cdot;\bar \delta)$ for any $\bar \delta \leq \delta \leq T$. Finally, arguing as in the proof of Theorem \ref{t:ap}, we deduce the pre-compactness of the sequence $(u^{\eps_n},V_{\eps_n}^{u^{\eps_n}})$ and the claim follows.

As a consequence of the previous claim,  for each fixed $\bs \ell \in \bb R^3_+$ and for each $\mc A \in \mc N(\bs \Gamma_{\bs \ell})$ we have $ \{  u \colon (u,V_\eps^u) \not \in \mc A \} \cap \mc K_{\bs \ell,\eps} =\emptyset$ for any $\eps$ small enough. Hence, as $\bb P_\eps \big(
\mc K^a_{\bs \ell, \eps} \big)=1$, 
 \begin{equation}
\label{lunga2}
\begin{split}
 \sup_{ \mc A\in \mc N(\bs\Gamma_{\bs\ell})} \varlimsup_{\eps\to 0}\, \eps\lambda_\eps\, \log\bb P_\eps \big(
(u,V_\eps^u) \notin \mc A \big) & \leq  \varlimsup_{\eps\to 0}\, \eps\lambda_\eps\, \log\bb P_\eps \big(
u \notin \mc K_{\bs \ell, \eps} \big) \\
& \leq \bigvee_{\beta \in \{ \mathrm{ b ,c ,d} \}}  \varlimsup_{\eps\to 0}\, \eps\lambda_\eps\, \log\bb P_\eps \big(
u \notin \mc K^\beta_{\bs \ell, \eps} \big) \; ,
\end{split}
\end{equation}
and it remains to estimate the probabilities on the right-hand side. By Proposition \ref{prop:1} we have
\begin{equation}
\label{lunga3}
\lim_{\bs \ell \to \infty}  \varlimsup_{\eps\to 0}\, \eps\lambda_\eps\, \log\bb P_\eps \big(
u \notin \mc K^{\mathrm b}_{\bs \ell, \eps} \big)=-\infty \; .
\end{equation}
Given $\alpha_2 \in (0, \frac{1}{4d})$ we pick $\gamma \in (\alpha_2,\frac{1}{4d})$ and observe that for $\ell_2\geq 1$ large enough we have 
\[ 
\zeta_k \bar{\delta}_k^{-\gamma}:= \frac{ \ell_2 {\bar \delta}^{1-\alpha_2}}{4 T^{1-\alpha_2}} \bar{\delta}_k^{\alpha_2-\gamma}  \geq \frac{ \ell_2 {\bar \delta}^{1-\alpha_2}}{4 T^{1-\alpha_2}} \bar{\delta}^{\alpha_2-\gamma} \geq C_0 \ell_0\;,
\]
with $C_0$ and $\ell_0$ as in Proposition \ref{t:omep}. By applying this proposition we get, for $\eps \lambda_\eps \leq 1$ and $\eps \leq \eps_0$,
\[
\begin{split}
\bb P_\eps  \big(
u \notin & \mc K^{\mathrm c}_{\bs \ell, \eps} \big)  \le \sum_{k \geq 0}  \bb P_\eps  \big(\omega^\infty (u;\bar{\delta}_k) > \zeta_k \big) \leq \sum_{k \geq 0} \rme^{- (\eps \lambda_\eps)^{-1} \zeta_k (C_0 \bar{\delta}_k^{\gamma})^{-1}} \\
 & = \sum_{k \geq 0} \rme^{- C \ell_2 (\eps \lambda_\eps)^{-1} \bar{\delta}_k^{\alpha_2-\gamma}} \leq \rme^{- C \ell_2(\eps \lambda_\eps)^{-1}  \bar{\delta}^{\alpha_2 -\gamma}}
  \sum_{k \geq 0} \rme^{- C  \big(  \bar{\delta}_k^{\alpha_2 -\gamma} -\bar{\delta}^{\alpha_2 -\gamma}\big) }
\\
 & \le \bar C \rme^{- C \ell_2(\eps \lambda_\eps)^{-1}  \bar{\delta}^{\alpha_2 -\gamma}}\;,
\end{split}
\]
for positive constants $C$ and $\bar{C}$ independent of $\eps \lambda_\eps \in (0,1)$, so that
 \begin{equation}
\label{lunga4}
\lim_{\bs \ell \to \infty}  \varlimsup_{\eps\to 0}\, \eps\lambda_\eps\, \log\bb P_\eps \big(
u \notin \mc K^{\mathrm c}_{\bs \ell, \eps} \big)=-\infty \; .
\end{equation}
 
Given $\alpha_3 \in (0, \frac 12)$ we pick $\alpha \in (\alpha_3,\frac 12)$ and observe that for $\ell_3\geq 1$ large enough we have, 
 \[ \zeta_k \bar{\delta}_k^{-\alpha}:= \frac{ \ell_3 {\bar \delta}^{1-\alpha_3}}{4 T^{1-\alpha_3}} \bar{\delta}_k^{\alpha_3-\alpha}  \geq \frac{ \ell_3 {\bar \delta}^{1-\alpha_3}}{4 T^{1-\alpha_3}} \bar{\delta}^{\alpha_3-\alpha} \geq C_2 \ell_0 \; ,\]
 with $C_2$ and $\ell_0$ as in Proposition \ref{t:ene}. By applying this proposition we get, for $\eps \lambda_\eps \leq 1$ and $\eps \leq \eps_2$,
 \[
 \begin{split}
 \bb P_\eps  \big(
u \notin & \mc K^{\mathrm d}_{\bs \ell, \eps} \big)  \le  \sum_{j=1}^{\lfloor (\eps \lambda_\eps)^{-1} \rfloor}\sum_{k \geq 0}  \bb P_\eps  \big(
 \omega^1 (|V^u_\eps|(\phi_j);\bar{\delta}_k) > \zeta_k \big)  \\
  & \leq  (\eps \lambda_\eps)^{-1}\sum_{k \geq 0} \rme^{- (\eps \lambda_\eps)^{-1}\frac{\zeta_k}{C_2 \bar{\delta}_k^{\alpha}}} = (\eps \lambda_\eps)^{-1}\sum_{k \geq 0} \rme^{- C \ell_3 (\eps \lambda_\eps)^{-1} \bar{\delta}_k^{\alpha_3-\alpha}} \\
  & \leq (\eps \lambda_\eps)^{-1} \rme^{- C \ell_3(\eps \lambda_\eps)^{-1}  \bar{\delta}^{\alpha_3 -\alpha}}
  \sum_{k \geq 0} \rme^{- C  \big(  \bar{\delta}_k^{\alpha_3 -\alpha} -\bar{\delta}^{\alpha_3 -\alpha}\big) }
\\
 & \le \bar C (\eps \lambda_\eps)^{-1}\rme^{- C \ell_3(\eps \lambda_\eps)^{-1}  \bar{\delta}^{\alpha_3 -\alpha}} \; ,
 \end{split}
 \]
 for positive constants $C$ and $\bar{C}$ independent of $\eps \lambda_\eps \in (0,1)$, so that
 \begin{equation}
\label{lunga5}
\lim_{\bs \ell \to \infty}  \varlimsup_{\eps\to 0}\, \eps\lambda_\eps\, \log\bb P_\eps \big(
u \notin \mc K^{\mathrm d}_{\bs \ell, \eps} \big)=-\infty \; .
\end{equation}
Gathering together \eqref{lunga2}-\eqref{lunga5} the bound \eqref{lunga} follows.
\qed\end{proof}

\subsection{Exponential martingales}

The upper bound will be achieved by a suitable exponential tilt of the probability $\bb P_\eps$. This tilt is constructed by means of families of martingales that are here introduced. 
\begin{lem}
\label{lem:2}
Given $\eta\in C^\infty([0,T]\times\bb T^d;\bb R^d)$ and $\psi\in C^\infty_K([0,T)\times\bb T^d)$ let $N^{1,\eta}$, $N^{2,\psi}$ be the $\bb P_\eps$-martingales defined by 
\begin{align}
\label{5}
N^{1,\eta}_t & := \eps\int_0^t\! \Big\langle \nabla u_s\cdot \eta_s\,,\, \rmd u_s - \Big(\Delta u_s - \frac 1{\eps^2} W'(u_s) \Big)\,\rmd s\Big\rangle_{L^2}\;,
\\ \label{6bb} N^{2,\psi}_t & := \int_0^t\! \Big\langle \psi_s\Big[-\eps\Delta u_s + \frac 1\eps W'(u)\Big]\,,\, \rmd u_s - \Big(\Delta u_s - \frac 1{\eps^2} W'(u_s) \Big)\,\rmd s\Big\rangle_{L^2}\;.
\end{align}
Then,
\begin{align}
\label{vva1}
N^{1,\eta}_T & = - J_\eps^u(\eta) - \int_0^T\! \delta V_{\eps,s}^u(\eta_s)\,\rmd s + \int_0^T\! \int\! \mathbf{n}^u_s \cdot D\eta_s\, \mathbf{n}^u_s\, \rmd \xi^u_{\eps,s}\,\rmd s\;, \\ \label{vva2}
\begin{split}
N^{2,\psi}_T & = - \mu^{\bar u_\eps^0}(\psi_0) - \int_0^T\! \mu^u_s(\partial_s\psi_s)\, \rmd s -J_\eps^u(\nabla\psi) - R^\psi_T  \\ & \quad + \frac 1\eps \int_0^T\!\int\!\psi \Big(\eps\Delta u - \frac 1\eps W'(u)\Big)^2 \,\rmd x\,\rmd s \;,
\end{split}
\end{align}
where $R^\psi_T$ is a random variable for which there exists a sequence $\zeta_\eps\to 0$ as $\eps\to 0$ such that 
\begin{equation}
\label{vva3}
\lim_{\eps\to 0}\, \eps\lambda_\eps\, \log\bb P_\eps\left(|R^\psi_T| > \zeta_\eps\right) = -\infty\;.
\end{equation} 
Finally, setting $N^{\eta,\psi}:= N^{1,\eta} + N^{2,\psi}$, its quadratic variation satisfies,
\begin{equation}
\label{nep}
\begin{split}
(\eps\lambda_\eps)^{-1} [N^{\eta,\psi}]_T  & \le  - 2 \int_0^T\! \delta V_{\eps,s}^u(2\psi_s\eta_s)\,\rmd s + 4 \int_0^T\! \int\! \mathbf{n}^u_s \cdot D(\psi_s\eta_s)\, \mathbf{n}^u_s\, \rmd \xi^u_{\eps,s}\,\rmd s \\ & \quad +2\int_0^T\!\int\! |\eta_s|^2 \, \rmd |V_{\eps,s}^u|\, \rmd s + 2\int_0^T\! \int\! |\eta_s|^2 \, \rmd \xi^u_{\eps,s}\,\rmd s \\ & \quad + 2\int_0^T\! \int\!  \frac 1\eps \psi^2\Big(\eps\Delta u - \frac 1\eps W'(u)\Big)^2  \, \rmd x\,\rmd s\;.
\end{split}
\end{equation}
\end{lem}

\begin{proof}
The equation \eqref{vva1} follows from the identity below (with $X=\eta_s$), which holds for any time $u\in H^1$ and vector field $X\in  C^\infty(\bb T^d;\bb R^d)$,
\begin{equation}
\label{id6}
\int\! \nabla u \cdot X \Big(\eps\Delta u - \frac 1\eps W'(u) \Big)\rmd x  = \delta V^u_\eps(X) - \int\! \mathbf{n}^u \cdot DX\, \mathbf{n}^u\, \rmd \xi^u_\eps\;,
\end{equation}
whose proof can be found in \cite{RS}. The representation \eqref{vva2} is deduced noticing that $N^{2,\psi}$ is the same martingale \eqref{6b} used in the proof of Lemma \ref{lem:J2}. Moreover, the bound \eqref{vva3} follows from \eqref{rtp} together with \eqref{2.5b} and Proposition \ref{prop:1}, for any $\zeta_\eps$ vanishing slower than $\eps\lambda_\eps\|\nabla j_\eps\|_{L^2}^2 +\eps^{-1}\lambda_\eps \|j_\eps\|_{L^2}^2$ as $\eps\to 0$.

We next observe that,
\[
\begin{split}
[N^{\eta,\psi}]_T & = 2\lambda_\eps \int_0^T\! \int\! \Big\{j_\eps * \Big[\eps\nabla u\cdot\eta- \psi \Big(\eps\Delta u-\frac 1\eps W'(u)\Big)\Big]\Big\}^2\, \rmd x\,\rmd s\;. \\ & \le  - 4\eps\lambda_\eps\int_0^T\! \int\! \nabla u \cdot (\psi\eta)  \Big(\eps\Delta u-\frac 1\eps W'(u)\Big)\, \rmd x\,\rmd s \\ & \quad + 2\eps\lambda_\eps\int_0^T\! \int\! \Big[\eps |\nabla u|^2 |\eta|^2 + \frac 1\eps \psi^2 \Big(\eps\Delta u-\frac 1\eps W'(u)\Big)^2\Big]\, \rmd x\,\rmd s\;,
\end{split}
\]
from which \eqref{nep} follows by applying \eqref{id6} with $X=\psi_s\eta_s$ to the first term in the right-hand side, and \eqref{dm} to the second one. 
\qed\end{proof}

Given $\bs\ell\in \bb R_+^3$, $\mc A\in \mc N(\bs\Gamma_{\bs\ell})$, $m>0$, and $N\in \bb N$, for $Z\in \bs Z$, we set,
\begin{equation}
\label{ip}
\mc I^{\eta,\chi,\psi}_{\bs\ell,\mc A,m,N}(Z) := \begin{cases} \mc I^{\eta,\chi,\psi}(Z) & \text{ if } Z \in \bs D_{\bs\ell,\mc A,m,N}, \\ +\infty & \text{ otherwise}. \end{cases}
\end{equation}
The following lemma which relies on the previous estimates, is the key step in the proof of the large deviations principle. 
\begin{lem}
\label{lem:up}
There exists a real sequence $a_{\bs\ell,m}\to +\infty$ as $\bs\ell,m\to +\infty$ such that the following holds. For each $\bs\ell\in \bb R_+^3$, $m>0$ , $\mc A\in \mc N(\bs\Gamma_{\bs\ell})$, $N\in \bb N$, $\delta>0$, each functions $\eta\in C^\infty([0,T]\times\bb T^d;\bb R^d)$, $\chi\in C^\infty_K([0,T)\times\bb T^d;\bb R^d)$, $\psi\in C^\infty_K([0,T)\times\bb T^d;[0,1))$ with $\supp(\chi)\subset\supp(\psi)$, and each Borel set $\bs B\subset \bs Z$, 
\[
\varlimsup_{\eps\to 0} \eps\lambda_\eps\log \bb P_\eps\big( Z_\eps^u \in \bs B \big) \le - \inf_{Z\in \bs B} \big\{\big[\mc I^{\eta,\chi,\psi}_{\bs\ell,\mc A,m,N}(Z)  - \delta \big] \wedge a_{\bs\ell,m}\big\}\;.
\]
\end{lem}

\begin{proof}
Let $N^{\eta,\psi,\eps}$ be the martingale introduced in Lemma \ref{lem:2} with $\eta$ and $\psi$ replaced by $(\eps\lambda_\eps)^{-1}\eta$ and $(\eps\lambda_\eps)^{-1}\psi$ respectively. By using the exponential martingale of $N^{\eta,\psi,\eps}$ we introduce the sub-probability, 
\begin{equation}
\label{pti}
\rmd \bb P_\eps^{\eta,\psi} := \rmd \bb P_\eps \exp \Big\{ N^{\eta,\psi,\eps}_T - \frac 12 [N^{\eta,\psi,\eps}]_T\Big\}\;.
\end{equation}
By using \eqref{vva1}, \eqref{vva2}, \eqref{nep}, and recalling \eqref{mcI},
\[
\begin{split}
 & N^{\eta,\psi,\eps}_T - \frac 12 [N^{\eta,\psi,\eps}]_T \ge (\eps\lambda_\eps)^{-1} \Big\{\mc I^{\eta,\chi,\psi}(Z_\eps^u) + \bar\mu_0(\psi_0) - \mu^{\bar u_\eps^0}(\psi_0) - R^\psi_T \\ & \qquad \quad +  \int_0^T\! \int\! \Big[\mathbf{n}^u_s \cdot D\big((1-2\psi_s)\eta_s\big)\, \mathbf{n}^u_s - |\eta_s|^2\Big]\, \rmd \xi^u_{\eps,s}\,\rmd s + \mc R^{\chi,\psi}(u) \Big\}\;,
\end{split}
\]
where
\[
\begin{split}
\mc R^{\chi,\psi}(u) & := \frac 1\eps \int_0^T\!\int\!\psi(1-\psi)\Big(\eps\Delta u - \frac 1\eps W'(u)\Big)^2 \,\rmd x\,\rmd s \\ & \quad - \int_0^T\! \delta V_s (\chi_s\sqrt{\psi_s(1-\psi_s)})\,\rmd s + \int_0^T\!\int\! \frac{|\chi_s|^2}4\, \rmd |V_{\eps,s}^u|\, \rmd s \;.
\end{split}
\]
Plugging $\lambda = \eps \nabla u\cdot \chi$ in the inequality
\[
\psi(1-\psi)\Big(\eps\Delta u - \frac 1\eps W'(u)\Big)^2 \ge \lambda \sqrt{\psi(1-\psi)}\Big(\eps\Delta u - \frac 1\eps W'(u)\Big) - \frac{\lambda^2}{4} \quad\forall\, \lambda\in\bb R\;, 
\]
we get,
\[
\begin{split}
& \frac 1\eps \int_0^T\!\int\!\psi(1-\psi)\Big(\eps\Delta u - \frac 1\eps W'(u)\Big)^2\,\rmd x\,\rmd s \\  & \ge \int_0^T\!\int\!  \nabla u\cdot \chi \sqrt{\psi(1-\psi)}\Big(\eps\Delta u - \frac 1\eps W'(u)\Big) \,\rmd x\,\rmd s  - \int_0^T\!\int\!\frac{|\nabla u|^2 |\chi|^2}{4}\,\rmd x\,\rmd s \;,
\end{split}
\]
which implies, by \eqref{id6} with $X=\chi_s\sqrt{\psi_s(1-\psi_s)}$ and definition \eqref{dm}, 
\[
\mc R^{\chi,\psi}(u) \ge \int_0^T\! \int\! \Big[-\mathbf{n}^u_s \cdot D\big(\chi_s\sqrt{\psi_s(1-\psi_s)}\big)\, \mathbf{n}^u_s - \frac{|\chi_s|^2}{4}\Big]\, \rmd \xi^u_{\eps,s}\,\rmd s\;.
\]

To prove the statement we observe,
\begin{equation}
\label{inclu}
\begin{split}
& \{Z_\eps^u \in \bs B\} \subset \Big\{Z_\eps^u \in \bs B\cap \bs D_{\bs\ell,\mc A, m,N}\;,\; \int_0^T \!\big\| \xi^u_{\eps,t}\big\|_\mathrm{TV} \, \rmd t \le \delta\,,\, R^\psi_T \le \zeta_\eps \Big\} \\ & \quad\qquad  \cup \big\{Z_\eps^u \notin \bs D_{\bs\ell,\mc A, m,N}\big\}\cup\Big\{u\colon \int_0^T \!\big\| \xi^u_{\eps,t}\big\|_\mathrm{TV} \, \rmd t > \delta\Big\} \cup \big\{u\colon R^\psi_T > \zeta_\eps\big\}\;, 
\end{split}
\end{equation}
with $R^\psi_T$ and $\zeta_\eps$ as in \eqref{vva2}-\eqref{vva3}. Letting $\zeta_\eps' = |\bar\mu_0(\psi_0) - \mu^{\bar u_\eps^0}(\psi_0)|+\zeta_\eps$ we bound,
\[
\begin{split}
& \bb P_\eps \Big( Z_\eps^u \in \bs B\cap \bs D_{\bs\ell,\mc A, m,N}\;,\; \int_0^T \!\big\| \xi^u_{\eps,t}\big\|_\mathrm{TV} \, \rmd t \le \delta\,,\, R^\psi_T \le \zeta_\eps \Big) \\ & \le \bb E_\eps^{\eta,\psi}\Big(\rme^{- (\eps\lambda_\eps)^{-1} \big(\mc I^{\eta,\chi,\psi}(Z_\eps^u) - C_{\eta,\chi,\psi}\delta  -  \zeta_\eps'\big)} \id_{\bs B\cap \bs D_{\bs\ell,\mc A, m,N}} (Z_\eps^u)\Big) \\ & \le \exp\Big\{-(\eps\lambda_\eps)^{-1} \inf_{Z\in \bs B\cap \bs D_{\bs\ell,\mc A, m,N}} \big[\mc I^{\eta,\chi,\psi}(Z)  - C_{\eta,\chi,\psi} \delta - \zeta_\eps' \big] \Big\} \\ & = \exp\Big\{-(\eps\lambda_\eps)^{-1} \inf_{Z\in \bs B} \big[\mc I^{\eta,\chi,\psi}_{\bs\ell,\mc A,m,N}(Z)  - C_{\eta,\chi,\psi} \delta  - \zeta_\eps' \big] \Big\} \;,
\end{split}
\]
where $\bb E_\eps^{\eta,\psi}$ denotes the expectation with respect to the measure $\bb P_\eps^{\eta,\psi}$ defined in \eqref{pti} and 
\[
C_{\eta,\chi,\psi} := \|D\big((2\psi-1)\eta+\chi\sqrt{\psi(1-\psi)}\big)\|_\infty +  \|\eta\|_\infty^2 + \frac 14 \big\||\chi|^2\big\|_\infty\;.
\]
By redefining $\delta$, the proof of the lemma is now achieved, in view of the inclusion \eqref{inclu}, by the previous bound, \eqref{vva3}, and Proposition~\ref{lem:zd}.
\qed\end{proof}

\subsection{Minimax}

By applying a minimax argument, we next optimize the bound in Lemma \ref{lem:up} and deduce the large deviations upper bound for compacts. 
\begin{lem}
\label{lem:in}
For each compact $\bs K\subset \bs Z$,
\begin{equation}
\label{qu}
\varlimsup_{\eps\to 0} \eps\lambda_\eps \log\bb P_\eps\big( Z_\eps^u\in \bs K\big) \le -\inf_{Z\in \bs K} \mc I(Z)\;.
\end{equation}
\end{lem}

\begin{proof}
First we notice that in view of Lemma~\ref{lem:up} for each open set $\bs A\subset \bs Z$ we have
\[
\varlimsup_{\eps\to 0} \eps\lambda_\eps \log\bb P_\eps\big( Z_\eps^u\in \bs A\big) \le -  \sup_{\eta,\chi,\psi}\sup_{\bs\ell}\sup_{\mc A,m,N,\delta} \inf_{Z\in \bs A} \big\{\big[\mc I^{\eta,\chi,\psi}_{\bs\ell,\mc A,m,N}(Z)  - \delta \big] \wedge a_{\bs\ell,m}\big\}\;.
\]
Notice that for each $\bs\ell$, $m$, $\mc A$, $N$, $\delta$, and each functions $\eta$, $\chi$, $\psi$ with $\supp(\chi)\subset\supp(\psi)$ the map $Z \mapsto \big[\mc I^{\eta,\chi,\psi}_{\bs\ell,\mc A,m,N}(Z)  - \delta \big] \wedge a_{\bs\ell,m}$ is continuous.
In view of the minimax lemma in \cite[App.~2, Lemmata 3.2 and 3.3]{KL} (notice that both the proofs hold true for compact sets in Hausdorff topological spaces), from the previous bound we deduce that \eqref{qu} holds with rate function 
\[
\mc I_0 =  \sup_{\eta,\chi,\psi}\sup_{\bs\ell}\sup_{\mc A,m,N,\delta}\big[\mc I^{\eta,\chi,\psi}_{\bs\ell,\mc A,m,N} - \delta \big] \wedge a_{\bs\ell,m}\;.
\]
It thus remains to prove that $\mc I_0 = \mc I$. We first take the supremum over $\delta>0$ and  $\mc A \in \mc N(\bs\Gamma_{\bs\ell})$. We get 
\[
\mc I_0 = \mc I_1 := \sup_{\eta,\chi,\psi} \sup_{\bs\ell,m,N} \mc I^{\eta,\chi,\psi}_{\bs\ell,m,N} \wedge a_{\bs\ell,m}\;,
\] 
where
\[
\mc I^{\eta,\chi,\psi}_{\bs\ell,m,N}(Z) = \begin{cases} \mc I^{\eta,\chi,\psi}(Z) & \text{ if } Z \in \bs D_{\bs\ell,m,N}, \\ +\infty & \text{ otherwise,} \end{cases}
\]
where $\bs D_{\bs\ell,m,N} := \bigcap_{\mc A\in \mc N(\bs\Gamma_{\bs \ell})}\bs D_{\bs\ell,\mc A,m,N}$. Taking $(f^{(k)})_{k\in \bb N}$ and $(\psi^{(k)})_{k\in \bb N}$ as at the beginning of Subsection \ref{subsec:4.1}, we let
\[
\begin{split}
\bs F_m & := \Big\{ (V,J)\in \bs V\times\bs H^{-\bs s}\colon \sup_k \Big[ J(f^{(k)}) -\frac 12\int_0^T\!\int |f^{(k)}|^2\, \rmd V_t\, \rmd t\Big]\le m \Big\}\;, \\ \bs G_m & := \Big\{ (V,J)\in \bs V\times\bs H^{-\bs s}\colon \sup_k \Big[ J(\nabla \psi^{(k)}) +\int_0^T\! V_t(\partial_t \psi^{(k)})\, \rmd t\Big] \le m \Big\}\;,
\end{split}
\]
and set
\begin{equation}
\label{Iboh2}
\begin{split}
&\bs D_{\bs\ell,m} := \bigcap_N  \bs D_{\bs\ell,m, N} \\
&= \Big\{ (u,V,J) \in \bs Z \;:\;\;  (u,V)\in \bs\Gamma_{\bs\ell}\;,\;\; (V,J)\in \bs F_m \cap \bs G_m\;,\;\; J(f)=0\;\; \forall\,f\in\bs L \Big\}\;.
\end{split}
\end{equation}
By taking the supremum over $N$, we deduce that 
\begin{equation}
\label{Iboh1}
\mc I_1 = \mc I_2 := \sup_{\eta,\chi,\psi} \sup_{\bs\ell,m} \mc I^{\eta,\chi,\psi}_{\bs\ell,m} \wedge a_{\bs\ell,m}\;,
\end{equation}
where
\begin{equation}
\label{Iboh1bis}
\mc I^{\eta,\chi,\psi}_{\bs\ell,m}(Z) := \begin{cases} \mc I^{\eta,\chi,\psi}(Z) & \text{ if } Z \in \bs D_{\bs\ell,m}, \\ +\infty & \text{ otherwise.} \end{cases}
\end{equation}

Since $(f^{(k)})_{k\in \bb N} \subset \bs H^{\bs s} \subset C([0,T] \times \bb T^d \times \Lambda_{d-1}; \bb R^d)$ with dense inclusions, if $(V,J) \in \bs F_m$ then $J$ extends by density to a continuous functional on $L^2(V_t\,\rmd t;\bb R^d)$ still denoted by $J$. By Riesz's representation lemma there exist $\nu\in L^2(V_t\,\rmd t;\bb R^d)$ such that
\begin{equation}
\label{Iboh3} 
J(f) = \int_0^T\!\int\! f\cdot \nu \, \rmd V_t \,\rmd t\;, \qquad f\in L^2(V_t\,\rmd t;\bb R^d)\;.
\end{equation}
We claim that $\bs D = \bigcup_{\bs\ell,m} \bs D_{\bs\ell,m}$, recall Definition \ref{def:d}. Clearly, the inclusion $\bs D \subset \bigcup_{\bs\ell,m} \bs D_{\bs\ell,m}$ holds by Definitions \ref{def:1}, \ref{def:d} and Remark \ref{rem:4}. To prove the other inclusion we first show that $Z\in \bs D_{\bs\ell,m}$ implies $\tau^{-1}V$ is an $L^2$-flow with velocity $\nu$. To this end, observe first that since $(u,V)\in \bs\Gamma_{\bs \ell}$, conditions a) and b) in Definition \ref{def:1} are fulfilled by Theorem \ref{t:ap}. Next, recalling the definition of $\bs L$ in Remark~\ref{rem:2}, the definition of $\bs D_{\bs \ell,m}$ in \eqref{Iboh2} and the representation \eqref{Iboh3}, for each $\eta\in C^\infty([0,T]\times \bb T^d;\bb R^d)$ we have,
\[
\int \! \nu_t(x) \cdot P_{\tau_x |V_t|}\eta_t(x) \,  |V_t|(\rmd x) \,\rmd t = 0\;,
\]
where $P_{\tau_x |V_t|}$ is the orthogonal projector onto the tangent plane to $|V_t|$ at the point $x$. This equation implies that $\nu_t(x) \perp \tau_x|V_t|$ for $|V_t|\,\rmd t$-a.e.\ $(t,x)$, i.e., the orthogonality condition \eqref{nut1} in Definition \ref{def:1}. Moreover, condition \eqref{nut2} is equivalent to the statement $(V,J)\in \bigcup_m\bs G_m$ in view of the density of $(\psi^{(k)})_{k\in \bb N}$ in the unit ball of $C_0((0,T)\times \bb T^d)$. We conclude that $\tau^{-1}V$ is $L^2$-flow with velocity $\nu$. Since, by Definition \ref{def:2}, $\bs \Gamma = \bigcup_{\bs\ell}\bs \Gamma_{\bs\ell}$, the inclusion $\bs D \supset \bigcup_{\bs\ell,m} \bs D_{\bs\ell,m}$ follows. 

The previous claim readily implies that 
\[
\sup_{\bs\ell,m} \mc I^{\eta,\chi,\psi}_{\bs\ell,m}(Z) = \begin{cases} \mc I^{\eta,\chi,\psi}(Z) & \text{ if } Z \in \bs D, \\ +\infty & \text{ otherwise.} \end{cases}
\]
Hence, by \eqref{mci} and \eqref{Iboh1}, $\mc I = \mc I_2$.
\qed\end{proof}

\subsection{Conclusion}

Given a sequence $\beta_\eps\downarrow 0$, we recall that a family of probabilities measures $\mc P_\eps$ on a Hausdorff topological space $\mc X$ is \emph{exponentially tight} with speed $\beta_\eps$ iff there exists a sequence of compacts $K_\ell\subset \mc X$ such that 
\[
\lim_{\ell\to +\infty}\varlimsup_{\eps\to 0} \beta_\eps \log \mc P_\eps\big(K_\ell^\complement\big) = -\infty\;.
\]

\begin{lem}
\label{lem:et}
The family of probabilities $(\bb P_\eps \circ (Z_\eps^u)^{-1})_{\eps>0}$ on $\bs Z$ is exponentially tight with speed $\eps\lambda_\eps$.
\end{lem}
\begin{proof}
We shall prove separately the exponential tightness of each variable. Concerning the compactness of $u$, for any $\ell>0$, as in the proof of Proposition \ref{lem:zd} we introduce the following subset of $\bs U$,
\[
K_\ell = \Big\{ u \in \bs U  \colon \sup_{t\in[0,T]} \|G(u_t) \|_{BV} \leq \ell\;, \;\;\omega^\infty(u; \bar{\delta}_k) \leq  \frac{ \ell {\bar \delta}^{1-\alpha_2}}{4 T^{1-\alpha_2}} \bar{\delta}_k^{\alpha_2} \;\;\forall\, k\in \bb N\Big\} \;, 
\]
where, as in Lemma \ref{t:ome}, $G(u)=\int_0^u \sqrt{2W(v)} \rmd v$ and $\bar{\delta}_k = \bar{\delta} 2^{-k}$.
Combining \eqref{3bb} and \eqref{lunga4} we have the estimate,
\[
\lim_{\ell\to +\infty}\varlimsup_{\eps\to 0} \eps\lambda_\eps \log \bb P_\eps\big(u\notin K_\ell \big) = -\infty\;.
\]
Moreover, arguing as in the proof of Theorem \ref{t:ap}, from the compact embedding $BV \hookrightarrow L^1$  and the equi-continuity of elements in $K_\ell$ as $\bar{\delta}_k \to 0$ we deduce that $K_\ell\subset \bs U$ is compact by Ascoli-Arzel\`a theorem.  

Recalling that $\bs V$ is equipped with the bounded weak* topology, norm bounded sets are precompact. Hence the exponential tightness of $V_\eps^u$ is a direct consequence of Proposition \ref{prop:1}, by choosing $K_\ell = \{V\in \bs V\colon \mathrm{ess\, sup}_{t\in[0,T]} \|V_t\|_{TV}\le\ell\}$ which is compact. 

We finally prove the exponential tightness of $J_\eps^u$. Given $\bs s\in (\frac 12,1)\times (\frac d2 ,+\infty)\times (\frac{d-1}2,+\infty)$ pick $\bs\sigma\in (\frac 12, s_1)\times (\frac d2,s_2)\times (\frac{d-1}2,s_3)$. By Sobolev embedding, bounded sets in $\bs H^{-\bs\sigma}$ are precompact in $\bs H^{-\bs s}$. Therefore, the tightness of $J_\eps^u$ follows from Lemma \ref{lem:J3}, by choosing $K_\ell = \{J\in \bs H^{-\bs s} \colon\|J\|_{\bs H^{-\bs\sigma}}^2\le\ell\}$ which is compact.
\qed\end{proof}

\noindent\textit{Proof of Theorem~\ref{thm:in}} \hspace{2truept} 
The exponential tightness in Lemma \ref{lem:et} together with the upper bound for compacts in Lemma~\ref{lem:in} imply the upper bound for closed sets by \cite[Lemma 1.2.18]{DZ}. It remains to prove the goodness of the rate function $\mc I$.  

Recall that, as shown in the proof of Lemma \ref{lem:in}, $\mc I=\mc I_2$, where $\mc I_2$ is defined in \eqref{Iboh1}. Let us first prove that, for each $\bs \ell$ and $m$, the set $\bs D_{\bs\ell,m}$ in \eqref{Iboh2} is compact. By Theorem \ref{t:ap}, $\bs \Gamma_{\bs\ell}$ is a compact subset of $\bs U\times \bs V$. Moreover, the sets $\bs F_m$ and  $\bs G_m$ are closed subsets of $\bs V\times \bs H^{-\bs s}$. Since the embedding $\mc M([0,T]\times \bb T^d\times\Lambda_{d-1};\bb R^d) \hookrightarrow \bs H^{-\bs s}$ is compact, the compactness of $\bs D_{\bs\ell,m}$ follows from a total variation upper bound for $J$. To this end, we observe that if $(u,V,J)\in \bs D_{\bs\ell,m}$ then the representation \eqref{Iboh3} gives, for each $f\in C([0,T]\times \bb T^d\times\Lambda_{d-1};\bb R^d)$, 
\[
|J(f)| \le \|\nu\|_{L^2(V_t\,\rmd t;\bb R^d)} \|f\|_{L^2(V_t\,\rmd t;\bb R^d)} \le m \essup_{t} \||V_t|\|_\mathrm{TV} \|f\|_\infty \le m \ell_1 \|f\|_\infty\;. 
\]
Since for each $\eta,\chi,\psi$ the functional $\mc I^{\eta,\chi,\psi}$ as defined in \eqref{mcI} is continuous, we get that the functional in \eqref{Iboh1bis} is lower semicontinuous. This implies the lower semicontinuity of $\mc I_2$. Finally, since $a_{\bs\ell,m}\to+\infty$, for each $q\in\bb R_+$ there exists $\bs\ell,m$ such that $\{\mc I_2 \le q\} \subset \bs D_{\bs\ell,m}$, which implies that $\{\mc I_2 \le q\}$ is pre-compact.
\qed

\appendix
\normalsize

\section{Measurability issues}
\label{app:a}
\begin{lem}
\label{mesura}
The map $C([0,T];L^2) \ni u \mapsto V_\eps^u \in \bs V$ defined by $V_{\eps,t}^{u} = V_\eps^{u_t}$, $t\in[0,T]$, for $u\in C([0,T];H^1)$, and $V_{\eps,t}^{u} = 0$, $t\in[0,T]$, otherwise, is Borel measurable. 
\end{lem}

\begin{proof}
First we note that $C([0,T];H^1)$ is a Borel subset of $C([0,T];L^2)$. We claim that, for any $f\in L^1([0,T]; C(\bb T^d \times \Lambda_{d-1}))$, the function 
\[
C([0,T];H^1) \ni u \mapsto \Gamma^f(u) := \int_0^T\! V_{\eps,t}^u (f_t) \,\rmd t \in \bb R
\]
is measurable with respect to the Borel $\sigma$-algebra of $C([0,T];L^2)$ restricted to $C([0,T];H^1)$. To prove this claim we introduce the following two-parameters approximation. Given two sequences $\delta_k\downarrow 0$ and $\eta_h\downarrow 0$, we set $R_k = (\mathrm{Id}-\delta_k\Delta)^{-1} \colon$ $ L^2\to H^2$ and $\phi_h\bb\colon \bb R \to\bb R$ be a continuous function such that $0\le\phi_h\le 1$, $\phi_h(\xi) = 0$ for $\xi\le 0$, and $\phi_h(\xi) = 1$ for $\xi\ge \eta_h$. We then define
\[
\begin{split}
\Gamma^f_{k,h}(u) & = \int_0^T\!\int\! \Big[ \phi_h(|\nabla R_k u_t|) f_t\Big(x,\frac{(\nabla R_ku_t)^\perp}{|\nabla R_ku_t|}\Big) \\ & \quad + (1-\phi_h(|\nabla R_k u_t|)) f_t\big(x,e_0^\perp\big) \Big] \mu^{R_ku_t}_\eps(\rmd x)\;\rmd t
\end{split}
\]
We note that $\Gamma^f_{k,h}$ as a function on $C([0,T];H^1)$ is continuous in the $C([0,T];L^2)$- topology. Moreover, as $k\to \infty$, $\Gamma^f_{k,h} \to \Gamma^f_h$ pointwise on $C([0,T];H^1)$, where
\[
\Gamma^f_h(u) = \int_0^T\!\int\! \Big[ \phi_h(|\nabla u_t|) f_t\Big(x,\frac{(\nabla u_t)^\perp}{|\nabla u_t|}\Big) + (1-\phi_h(|\nabla u_t|)) f_t\big(x,e_0^\perp\big) \Big] \mu^{u_t}_\eps(\rmd x)\;\rmd t\;.
\]
In particular, the map $C([0,T];H^1) \ni u \to \Gamma^f_h(u)$ is measurable with respect to the Borel $\sigma$-algebra of $C([0,T];L^2)$ restricted to $C([0,T];H^1)$. By dominated convergence, as $h\to\infty$,
$\Gamma^f_h \to \Gamma^f$ pointwise on $C([0,T];H^1)$, hence the claim follows.

In order to prove the required measurability, we write $C([0,T];H^1) = \bigcup_{\ell\in \bb N} \mc C_\ell$, where $\mc C_\ell := \{u\in C([0,T];H^1)\colon \|u\|_{C(H^1)}\le \ell\}$ and, similarly, $\bs V = \bigcup_{m \in \bb N} \bs V_m$ where $\bs V_m := \{V\in \bs V \colon \mathrm{ess} \sup_t\|V_t\|_{TV} \le m \}$. Clearly it is enough to show that each restriction of $\mc C_\ell \ni u \mapsto V_\eps^u\in \bs V$ is measurable. Next, we notice that, by Sobolev embedding, $|V_{\eps,t}^u|_{TV} = |\mu^u_{\eps,t}|_{TV} \le C_\eps (1+\|u_t\|_{H^1}^4)$, hence if $u\in \mc C_\ell$ then there exists $m^*=m^*(\eps,\ell)$ such that $V_\eps^u \in \bs V_{m^*}$. It is therefore enough to show the Borel measurability of the map $\mc C_\ell \ni u \mapsto V_\eps^u\in \bs V_{m^*}$. Since $\bs V_{m^*}$ is endowed with the weak* topology induced by the duality with the separable Banach space $L^1([0,T];C(\bb T^d\times \Lambda_{d-1}))$ then the topology of $\bs V_{m^*}$ has a countable basis and it is a compact metric space. Therefore, by definition of weak* topology and the initial claim the statement follows. 
\qed\end{proof}

\section{Deterministic bounds}
\label{app:b}

\noindent\textit{Proof of Theorem \ref{t:ap}} \hspace{2truept} 
We start by proving item a). The first  statement, i.e., $u \in L^\infty((0,T); BV(\mathbb{T}^d));\{\pm 1 \})$, follows from the strong convergence in $\mathbf{U}$ and the static result from \cite{Modica}, which yields the estimate $\tau \| u_t\|_{TV} \leq 2 \liminf_{\eps } \mathcal{F}_\eps(u_t^\eps)$ for a.e.\ $t \in (0,T)$. Finally, the last statement follows from the lower semicontinuity of $\omega^\infty(\cdot,\delta)$. 

To prove item b), first we observe that the bound $\textrm{ess} \sup_t \| |V_t| \|_{TV} \leq \ell_1$ follows readily from the weak* lower semicontinuity of the norm in $\bs V$ and condition b) in Definition~\ref{def:2}. Now we prove b.2). We fix  a vector field $\eta \in C^1( [0,T] \times \mathbb{T}^d; \mathbb{R}^d)$ and, as in \eqref{id6}, we write,
\[  
\begin{split}
\int_0^T\! \delta V_{\eps,t}^{u^\eps} (\eta_t)\, \rmd t & = \int_0^T\!  \int\! \nabla u^\eps_t\cdot \eta_t \Big(\eps\Delta u_t^\eps - \frac 1\eps W'(u_t^\eps) \Big)\rmd x \, \rmd t \\ & \quad - \int_0^T \!\int\! \mathbf{n}^{u^\eps}_t \cdot D\eta_t\, \mathbf{n}^{u^\eps}_t\, \rmd \xi^{u^\eps}_t \rmd t\;.
\end{split}
 \]
By Lemma \ref{prop:2}, Cauchy-Schwartz inequality, and condition b) in Definition~\ref{def:2}, as $\eps \to 0$ we have,
\[   
\Big| \int_0^T\! \delta V_t (\eta_t) \, \rmd t \Big| \leq  \ell_1^{1/2}  \Big(\int_0^T\!  
 |V_t| (|\eta|^2) \, \rmd t \Big)^{1/2}\; .
 \] 
By density, we can apply the Riesz representation theorem to obtain $\int_0^T \! \delta V_t(\eta_t)\, \rmd t$ $= - \int_0^T \! |V_t|(H \cdot \eta_t) \, \rmd t$ for some $H \in L^2 ([0,T] \times \bb T^d, |V_t| \, \rmd t ; \bb R^d )$ satisfying the bound in b.2) and for any vector field $\eta \in C^1( [0,T] \times \mathbb{T}^d; \mathbb{R}^d)$. By Fubini theorem we conclude that for a.e.\ $t\in [0,T]$ the varifold $V_t$ has bounded first variation represented by the mean curvature vector $H_t \in L^2( \bb T^d, |V_t|; \bb R^d)$. The proof of b.2) is thus completed.

In order to prove b.1), let $\{ \phi_j\}$ be the dense subset in the unit ball of $C^1(\mathbb{T}^d)$ in Definition \ref{def:2}. In view of conditions b) and d) in Definition~\ref{def:2}, by the Kolmogorov-Riesz-Fr\'echet compactness criterion, we can pass to a subsequence so that there exists $ \lim_\eps |V_t^{u^\eps}|(\phi_j)$  in $ L^1([0,T])$ and for a.e.\ $t\in (0,T)$ for any $j \in \bb N$. Moreover, in view of the uniform mass bound b) the same holds for any $\phi \in C(\mathbb{T}^d)$ by density and homogeneity. On the other hand, since $(|V^{u^\eps}|)_{\eps>0} \subset L^\infty([0,T]; \mc M_+)$ is uniformly bounded, up to subsequence $|V^{u^\eps}| \to \mu$ weakly* for some $\mu \in L^\infty([0,T]; \mc M_+)$ and $\mathrm{ess\,sup}_{t\in [0,T]} \,\| \mu_t\|_{TV} \leq \ell_1$. Thus, by dominated convergence $\mu_t(\phi)=\lim_\eps |V_t^{u^\eps}|(\phi)$ for a.e. $t \in [0,T]$, i.e., up to subsequences, $\mu_t=\lim_\eps |V_t^{u^\eps}|$ weakly as measures for a.e.\ $t\in (0,T)$.  By condition b) in Definition~\ref{def:2} and Fatou's lemma we have $\int_0^T \varliminf_{\eps} \mathcal{W}_\eps(u^\eps_t) \, \rmd t \leq  \ell_1$, hence for a.e.\ $t\in [0,T]$, possibly passing to a further subsequence depending on $t$, we have that $\lim_\eps \mathcal{W}_\eps(u^\eps_t) <+\infty$. Applying \cite[Thms.~4.1, 5.1]{RS} we deduce that, a.e.\ $t\in [0,T]$, $\mu_t$ is rectifiable and $\tau^{-1} \mu_t$ is an integral measure.  By b.2) for a.e.\ $t\in[0,T]$ the varifold $V_t$ has bounded first variation and $\tau^{-1}|V_t|$ is an integral measure, thus $\tau^{-1}V_t$ is an integral varifold by Allard rectifiability theorem (see, e.g., \cite[Thm.~42.4]{Simon}). Finally, item b.3) follows from the lower semicontinuity of $\omega^1(\cdot,\delta)$.
 
Item c) follows easily from the inequality $| \nabla G(u^\eps_t)| \rmd x \leq \mu^{u^\eps_t}(\rmd x)$ (i.e., the standard trick from \cite{Modica}) with $G(u) := \int_0^u \! \sqrt{2W(v)}\, \rmd v$ as in Lemma~\ref{t:ome}, the lower semicontinuity of $BV$-norm with respect to $L^1$-convergence and item a), recalling that $\tau=G(1)-G(-1)$.

It remains to show  the compactness of $\bs \Gamma_{\bs \ell}$. By properties a) and b) proven above, $\bs \Gamma_{\bs \ell}$ is a norm bounded subset of $\mathbf{U} \times \mathbf{V}$. Hence, as recalled in Section \ref{sec:2}, it is metrizable so that it is compact iff it is sequentially compact. Let $((u_n,V_n)) \subset \bs \Gamma_{\bs \ell} $.  In view of property a) and the compact embedding $BV \hookrightarrow L^1$, the sequence $(u_n) \subset \bs U$ is precompact by Ascoli-Arzela theorem. On the other hand $(V_n) \subset \bs V$ is precompact in view of the uniform mass bound in b). Thus, up to subsequences $(u,V)=\lim_n (u_n,V_n)$ and it remains to show that $(u,V) \in \bs \Gamma_{\bs \ell}$. This follows easily by constructing a diagonal sequence $(u^\eps,V_\eps^{u^\eps})_{\eps>0}$, $(u^\eps)_{\eps>0}\subset C([0,T];H^1)\cap L^2([0,T];H^2)$,  from the approximating sequences for each $(u_n,V_n)$ which keeps the conditions a)-d) in Definition \ref{def:2}.
\qed

\section{Stochastic currents}
\label{app:c}

Let us first briefly review the theory of It\^o stochastic currents for semimartingales in $\bb R^n$ as developed, e.g., in \cite{FGGT}. Let $(X_t)_{t \in [0,T]}$ be a continuous semimartingale on $\bb R^n$ and $f \colon \bb R^n \to \bb R^n$ a smooth vector field with compact support. Then the It\^o stochastic integral 
$\mc J(f):=\int_0^T f(X_t) \cdot dX_t$ is well defined with probability one. Since the exceptional set depends on $f$, it is not obvious that, with probability one, the map $f \mapsto \mc J(f)$ extends to a continuous linear functional on a suitable functional space for the vector field $f$. This issue is solved in 
\cite[Thm.~9]{FGGT}, where it is shown that, with probability one, $f \mapsto \mc J(f)$ defines a continuous linear functional on $H^s(\bb R^n;\bb R^n)$ for $s>n/2$.

Here, we develop a theory of stochastic currents for the processes obtained by solving the stochastic Allen-Cahn equation \eqref{1}. We do not attempt a theory of infinite dimensional currents but we define them on a restricted class of vector fields that are sufficient for our purposes. This analysis does not depend on the scaling parameters $\eps$ and $\lambda_\eps$, therefore, to simplify the notation, throughout this section we set $\eps=\lambda_\eps=1$ and drop them from the notation.

As proven in \cite{BBP1}, given $\bar u_0\in H^1$ and $T>0$ there exists a unique strong solution to \eqref{1} with initial condition $\bar u_0$. Moreover, denoting by $\bb P$ the induced law on $\Omega:=C([0,T];L^2)$, it satisfies $\bb P (u \in C([0,T];H^1)\cap L^2([0,T];H^2))=1$ and for  $p\in [1,\infty)$ there exists $C=C(\bar{u}_0,T,p)>0$ such that 
\begin{equation}
\label{steu}
\bb E \Big( \sup_{t\in [0,T]} \mc F (u_t) + \int_0^T\! \mc W (u_t)\, \rmd t \Big)^p \le C \;.
\end{equation}
Given $\bs H^{\bs s} = H^{s_1}([0,T];H^{s_2}(\bb T^d;H^{s_3}(\Lambda_{d-1};\bb R^d)))$, with $\bs s = (s_1,s_2,s_3) \in (\frac 12,1)\times (\frac d2 ,+\infty) \times (\frac{d-1}2,+\infty)$ and $f\in \bs H^{\bs s}$, the definition \eqref{J} reads,
\begin{equation}
\label{J2}
J^u(f) := - \int_0^T\! \big\langle\nabla u_t\cdot f_t(\cdot,(\mathbf{n}^u)^\perp)\big\rangle_{L^2}\, \rmd u_t\;,
\end{equation}
where we recall that $\mathbf{n}^u$ has been defined in \eqref{va}.

\begin{theorem}
\label{thm:sc}
Given $\bs s \in (\frac 12, 1) \times (\frac d2,+\infty) \times (\frac{d-1}2,+\infty)$, there exists a measurable map $ \Theta \colon \Omega \to \bs H^{-\bs s}$ such that $\bb P$-a.s.\ $\langle \Theta(u), f \rangle= J^u(f)$ for all $f \in \bs H^{\bs s}$. 
\end{theorem}

\begin{proof} 
Up to isometries, $\Lambda_{d-1}(\bb R^d) = \Lambda_1(\bb R^d) = S^{d-1}/\{\pm 1\}$, hence we can identify $H^s(\Lambda_{d-1};\bb R^d) = H^s_\mathrm{even}(S^{d-1};\bb R^d)\subset H^s(S^{d-1};\bb R^d)$, the closed subspace of even vector fields. We recall that for $s>\frac 12$ there exists a bounded linear extension operator $\mathrm{Ext} \colon H^s(S^{d-1};\bb R^d) \to H^{s+\frac 12}(\bb R^d;\bb R^d)$. Let $\bs s'= (s_1,s_2,s_3+\frac 12)$ and set
\[
\mc H^{\bs s'} := H^{s_1}([0,T];H^{s_2}(\bb T^d;H^{s_3+\frac 12}(\bb R^d;\bb R^d)))\;.
\]
With a slight abuse of notation, we also denote by $\mathrm{Ext}\colon \bs H^{\bs s} \to \mc H^{\bs s'}$ the bounded operator induced by the extension operator above. Note that, in view of the choice of $\bs s$, there is a continuos embedding $\mc H^{\bs s'}\subset C_0([0,T]\times\bb T^d\times\bb R^d;\bb R^d)$. Hereafter, for $f \in \bs H^{\bs s}$ we set $g:=\mathrm{Ext}(f)$, so that $g\in \mc H^{\bs s'}$. 

It is convenient to characterize the elements of $\mc H^{\bs s'}$ throughout their Fourier expansion. To this purpose, we introduce the functions $e_{n,k,q}^m \colon [0,T]\times \bb T^d \times \bb R^d \to \bb C^d $ defined by
\begin{equation}
\label{eq:enk}
e_{n,k,q}^m(t,x,p) := \frac{2-\delta_{n,0}}{\sqrt T} \cos\Big(\frac{n\pi t}T\Big) \,\rme^{2\pi\rmi\, k\cdot x} \, \frac{\rme^{\rmi q\cdot p}}{(2\pi)^{d/2}} \, e_m \;, 
\end{equation}
where $n\in \bb Z_+$, $k\in \bb Z^d$, $q \in \bb R^d$, and $e_1,\ldots, e_d$ is the canonical basis in $\bb R^d$. For $g$ as above we denote by
\[
\widehat g^m_n(k,q):= \int_0^T\!\int\! \int \big[e_{n,k,q}^m\big]^* \cdot g \,\rmd p\,\rmd x\, \rmd t
\]
its Fourier coefficients, where $*$ denotes complex conjugation. We remark that, by extending $g$ to an even function of $t\in [-T,T]$ and expanding it as a Fourier series, an equivalent norm in $\mc H^{\bs s'}$ is given by 
\begin{equation}
\label{normf}
|\!|\!| g|\!|\!|_{\bs s'}^2 = \sum_{m,n,k} \int\! (1+n^2)^{s_1} (1+|k|^2)^{s_2}(1+ |q|^2)^{s_3+\frac 12} |\widehat g^m_n(k,q)|^2\, \rmd q\;.
\end{equation}
The dual space $(\mc H^{\bs s'})^\prime$ can be identified with $\mc H^{- \bs s'}$ under the natural $L^2$-pairing $\langle \cdot, \cdot \rangle$ of the Fourier coefficients. 

Since $g=\mathrm{Ext}(f)$, \eqref{J2} reads,
\begin{equation}
\label{J1}
J^u(f) = -\int_0^T\! \big\langle \nabla u_t \cdot g_t\big( \cdot,\tfrac{\nabla u_t}{|\nabla u_t|}\big)\;, \rmd u_t \big\rangle_{L^2}\;.
\end{equation}

We observe that $f\mapsto J^u(f)$ is a linear map from $\bs H^{\bs s}$ to the measurable functions of $u$. Now, we claim that there exists a random constant $\mc C= \mc C(u)$ such that $\mc C \in L^2(\Omega;\rmd \bb P)$ and $|J^u(f)|\le \mc C \|f\|_{\bs H^{\bs s}}$, $f \in \bs H^{\bs s}$. Postponing the proof of the claim, we first show how this implies the existence of the map $\Theta$. 

We present  below a direct construction which is alternative to the abstract results in the literature, see, e.g., \cite[Lemma 2.2]{Flandoli}. Consider the map $\mc B$, acting on the set of simple functions on $\Omega$ taking value in $\bs H^{\bs s}$, defined by setting 
\[
\mc B(\Phi) = \bb E \Big(\sum_i J^u(f_i) \chi_{\Omega_i}\Big)\;,
\]
where $\Phi=\sum_i f_i \chi_{\Omega_i}$ with $(\Omega_i)$ a finite measurable partition of $\Omega$ and $f_i\in \bs H^{\bs s}$. From the claim and the Cauchy-Schwartz inequality,
\[
\begin{split}
|\mc B(\Phi) | & \le  \bb E \Big(\sum_i \mc C \|f_i\|_{\bs H^{\bs s}} \chi_{\Omega_i}\Big)\le \|\mc C\|_{L^2(\Omega)}\Big[\bb E \Big(\sum_i \|f_i\|^2 \chi_{\Omega_i}\Big)\Big]^{1/2} \\ & = \|\mc C\|_{L^2(\Omega)} \|\Phi\|_{L^2(\Omega;\bs H^{\bs s})}\;.
\end{split}
\]
Therefore, $\mc B$ is linear and bounded, whence it extends by density to $L^2(\Omega;\bs H^{\bs s})$. Since $\big(L^2(\Omega;\bs H^{\bs s})\big)' = L^2(\Omega;\bs H^{-\bs s})$, there is a unique $\Psi\in L^2(\Omega;\bs H^{-\bs s})$ such that, for any measurable subset $\Omega'$ of $\Omega$ and any $f\in \bs H^{\bs s}$,
\[
\bb E \big(\langle \Psi,f\rangle \chi_{\Omega'}\big) =  \mc B(f\chi_{\Omega'})\;.
\]
As $\mc B(f\chi_{\Omega'}) = \bb E \big(J^u(f)\chi_{\Omega'}\big)$, by the arbitrariness of $\Omega'$ it follows that $\bb P$-a.s.\ $\langle \Psi,f\rangle = J^u(f)$ for any $f\in \bs H^{\bs s}$. Choosing $\Theta\colon \Omega\to\bs H^{-\bs s}$ as any  representative of $\Psi$, we have that $\bb P$-a.s.\ $\langle \Theta(u),f\rangle = J^u(f)$. 

It remains to prove the claim. To this end, we write $J^u(f) = A_T^f+N_T^f$ where,
\[
A_T^f := -\int_0^T\! \Big\langle \nabla u_t \cdot g_t\big( \cdot,\tfrac{\nabla u_t}{|\nabla u_t|}\big)\,,(\Delta u_t-W'(u_t)) \,\Big\rangle_{L^2} \, \rmd t 
\]
and
\[
N_T^f := - \int_0^T\! \Big\langle  \nabla u_t\cdot g_t\big( \cdot,\tfrac{\nabla u_t}{|\nabla u_t|}\big)\;, \,\sqrt2\,\rmd\alpha_t \Big\rangle_{L^2}\;.
\]
By setting 
\begin{equation}
\label{mc1}
\mc C_1 := \int_0^T\! \int\! |\nabla u_t|^2\,\rmd x\,\rmd t \int_0^T\, \mc W(u_t)\,\rmd t \;,
\end{equation}
from Cauchy-Schwartz inequality we get,
\[
|A_T^f|^2\le \mc C_1 \|g\|_\infty^2  \le C \mc C_1 |\!|\!|g|\!|\!|^2_{\bs s'} \le C \mc C_1 \|f\|_{\bs H^{\bs s}}^2\;.
\]
By \eqref{steu}, the random constant $\mc C_1$ is such that $\mc C_1 \in L^1(\Omega;\rmd \bb P)$.

To analyze the martingale part $N_T^f$, we first observe that, as follows from Fourier inversion formula for $g$ and the stochastic Fubini's theorem (see also \cite[Lemma 8]{FGGT}), that $\bb P$-a.s.
\begin{equation}
\label{quella}
N_T^f = \sum_{m,n,k} \int\! \widehat g_n^m \,(k,q)^* Z_n^m(k,q) \, \rmd q\;,
\end{equation}
where $Z_n^m(k,q)$ is the complex random variable 
\[
Z_n^m(k,q) = \int_0^T\! \Big\langle \nabla u_t \cdot e_{n,k,q}^m\big(t,\cdot,\tfrac{\nabla u_t}{|\nabla u_t|}\big)\;, \sqrt 2\,\rmd \alpha_t \Big \rangle_{L^2} \;.
\]
By setting 
\begin{equation}
\label{mc2}
\mc C_2 := \sum_{m,n,k} \int\! (1+n^2)^{-s_1} (1+|k|^2)^{-s_2}(1+|q|^2)^{-s_3-\frac 12} |Z_n^m(k,q)|^2\, \rmd q \;,
\end{equation}
from Cauchy-Schwartz inequality in \eqref{quella} we get,
\[
|N_T^f|^2 \le \mc C_2 |\!|\!|g|\!|\!|_{\bs s'}^2 \le  C \mc C_2 \|f\|_{\bs H^{\bs s}}^2\;. 
\]
The random constant $\mc C_2$ is such that $\mc C_2 \in L^1(\Omega;\rmd \bb P)$. In fact, by a straightforward computation and using again \eqref{steu},
\[
\begin{split}
\bb E |Z_n^m(k,q)|^2 & = 2\bb E \int_0^T\!  \Big\| j*\Big(\nabla u_t\cdot e_{n,k,q}^m\big(t,\cdot,\tfrac{\nabla u_t}{|\nabla u_t|}\big)\Big)\Big\|_{L^2}^2 \, \rmd t \\ & \le 2 
\bb E \int_0^T\!  \Big\| \nabla u_t\cdot e_{n,k,q}^m\big(t,\cdot,\tfrac{\nabla u_t}{|\nabla u_t|}\big) \Big\|_{L^2}^2 \, \rmd t \\
& \le 2\|e_{n,k,q}^m\|_\infty^2\, \bb E \int_0^T\! \int\! |\nabla u_t|^2\,\rmd x\, \rmd t  \le C\;,
\end{split}
\]
for some $C$ depending only on $d$ and $T$. Since $\bs s\in (\frac 12, 1) \times (\frac d2,+\infty) \times (\frac{d-1}2,+\infty)$, we then get
\[
\bb E\, \mc C_2 \le C \sum_{n,k} \int\! (1+n^2)^{-s_1} (1+|k|^2)^{-s_2} (1+|q|^2)^{-s_3-\frac 12}\, \rmd q <\infty\,. 
\]

By the previous estimates, the claim is thus proven with $\mc C = C \sqrt{\mc C_1 + \mc C_2}$. 
\qed\end{proof}

\begin{rem}
\label{rem:boh}
In the proof of the previous theorem, we actually proven the estimate,
\[
\|J^u\|_{\bs H^{-\bs s}}^2 \le C(\mc C_1 + \mc C_2)\;,
\]
where the $\bb P$-a.s.\ finite random constant $\mc C_1 = \mc C_1(u)$ and $\mc C_2 = \mc C_2(u)$ are defined in \eqref{mc1} and \eqref{mc2} respectively. 
\end{rem}

\end{document}